\documentclass[reqno,11pt]{amsart}
\usepackage{amssymb}
\usepackage[mathscr]{eucal}
\usepackage{eufrak}
\usepackage{graphicx}
\usepackage{float}
\usepackage{subfigure}
\usepackage{xspace}
\usepackage{tikz}
\usepackage{multirow}
\usepackage{rotating,makecell}
\usepackage[pagewise]{lineno}

\newcommand{\dd}{\mathrm{d}}

\newcommand{\D}{\mathbb{\D}}

\newcommand{\LL}{\mathcal{L}}

\theoremstyle{plain}
\newtheorem{theorem}{Theorem}[section]

\newtheorem{lemma}{Lemma}[section]

\theoremstyle{definition}
\newtheorem{definition}{Definition}[section]
\theoremstyle{remark}

\newtheorem{remark}{Remark}[section]
\numberwithin{equation}{section}
\numberwithin{figure}{section}

\begin{document}
\title[Hypersonic-Limit Euler Flows and Measure Solutions] {On Two-Dimensional Steady Hypersonic-Limit Euler Flows Passing Ramps and Radon Measure Solutions of Compressible Euler Equations}

\author{Yunjuan Jin}
\author{Aifang Qu}
\author{Hairong Yuan}

\address[Y. Jin]{Center for Partial Differential Equations, School of Mathematical Sciences,
East China Normal University, Shanghai
200241, China}
\email{\tt 52185500019@stu.ecnu.edu.cn}

\address[A. Qu]{Department of Mathematics, Shanghai Normal University,
Shanghai,  200234,  China} \email{\tt afqu@shnu.edu.cn, aifangqu@163.com}

\address[H. Yuan]{Corresponding author. School of Mathematical Sciences and Shanghai Key Laboratory of Pure Mathematics and Mathematical Practice,
East China Normal University, Shanghai
200241, China}
\email{\tt hryuan@math.ecnu.edu.cn}

\date{\today}

\begin{abstract}
We proposed  rigorous definitions of Radon measure solutions for boundary value problems of steady compressible Euler equations which modeling hypersonic-limit inviscid flows passing two-dimensional ramps, and their interactions with still gas and pressureless jets.   We proved the Newton-Busemann pressure law of drags of body in hypersonic flow, and constructed various physically interesting measure solutions with density containing Dirac measures supported on curves, also exhibited examples of blow up of certain measure solutions. This established a mathematical foundation for applications in engineering and further studies of measure solutions of compressible Euler equations.
\end{abstract}
\maketitle

\tableofcontents

\section{Introduction} \label{sec1}

In gas dynamics, supersonic flow with Mach number greater than five is called hypersonic flow, which bears some peculiar features \cite[Section 15.2]{Anderson1}. For example, as the Mach number of the flow goes to infinity, it behaves like moving particles without thermal motions (hence pressure approaches zero); when passing a slender body, {\it shock layer, } i.e., the region bounded by the surface of the body and the shock appeared in front of it, becomes thinner and thinner, and ultimately mass concentrates in an infinite-thin shock layer. There is also a {\it Mach number independence law} (see \cite[Section 15.5]{Anderson1} or \cite[p.24]{Hayes2004hypersonic}), which claims that two flow fields with large but different upstream flow Mach numbers are not different from each other in any fundamental way.

These physical observations imply that there is a limit of hypersonic flow problem, and the limiting flow field cannot be described by Lebesgue measurable functions anymore. Actually a suitable concept of measure solutions of the compressible Euler equations should be introduced to illustrate the above observations and put related physical arguments upon a solid mathematical foundation. However, it is somewhat surprising that there is no such mathematical work before the paper \cite{Qu2018Hypersonic}. In \cite{Qu2018Hypersonic} the  authors showed that for steady Euler flows of polytropic gases, after suitable scaling, the Mach number goes to infinity means actually that the adiabatic exponent goes to $1$. By proposing a definition of measure solutions for supersonic flow passing a two-dimensional straight wedge, the authors verified the above physical observations mathematically, and derived naturally the Newton's sine square pressure law. See \cite{Anderson1,Anderson2006hypersonic,Hayes2004hypersonic,LO1991} for the background and physical theory of hypersonic flows, especially \cite[Chapter 3]{Hayes2004hypersonic} for a detailed introduction to the Newton's theory of hypersonic flow. In \cite{Qu2018High,Qu2018measure}, the authors also studied the related problems of measure solutions of high Mach number limits of piston problems for polytropic gases and Chaplygin gas. These papers demonstrate that the concept of measure solutions we proposed works well for these fundamental physical problems.

In this paper we are going to study three typical problems about limiting hypersonic flows passing bodies, with emphasis on explicit solutions derived from rigorous mathematical theory. The first is on hypersonic-limit flow passing a curved ramp and to derive the Newton-Busemann pressure law \cite[p.133]{Hayes2004hypersonic}. The second is to study the interactions of hypersonic limit flow and still gas in a ``dead gas zone". The third is on limiting hypersonic flow interacts  with pressureless jets. For the latter two problems,  {\it free layer} (called ``delta shock" in the literature of mathematics) appears in the flow field that separates gases with different states.
We calculate special measure solutions to understand these physical problems, and find some interesting new phenomena, such as blow up of measure solutions in a finite distance,  which  exhibit the great  power of a proper concept of measure solutions to the Euler equations.

In the rest of this section we firstly present  the problems, and the concept of measure solutions, as well as the main results. After that, we review briefly some related mathematical works on measure solutions of hyperbolic conservation laws, at the end of the section.

\subsection{Formulation of three problems}\label{sec11}

In the Euclidean plane $\mathbb{R}^2$ with Cartesian coordinates $(x,y)$, the   two-dimensional steady non-isentropic compressible Euler equations take the form \cite[Section 6.2]{Anderson1}
\begin{equation}\label{1.1}
\left \{
\begin{aligned}
&\partial_{x}(\rho u)+\partial_{y}(\rho v)=0,\\
&\partial_{x}(\rho u^{2}+p)+\partial_{y}(\rho uv)=0,\\
&\partial_{x}(\rho uv)+\partial_{y}(\rho v^{2}+p)=0,\\
&\partial_{x}(\rho uE)+\partial_{y}(\rho vE)=0,\\
\end{aligned} \right.
\end{equation}
which can also be written as
\begin{equation}\label{1.4}
\partial_{x}F(U)+\partial_{y}G(U)=0, \ \  U=(\rho,u,v,E)^{\top},
\end{equation}
where
$$F(U)=(\rho u,\ \rho u^{2}+p,\ \rho uv,\ \rho uE)^{\top},$$
$$G(U)=(\rho v,\ \rho uv,\ \rho v^{2}+p,\ \rho vE)^{\top}.
$$
Here $\rho,p,(u,v)$ represent respectively the density of mass, pressure, and velocity of the gas; the function $E$ is given by
\begin{equation}\label{1.2}
E=\frac{1}{2}(u^{2}+v^{2})+\dfrac{\gamma}{\gamma-1}\dfrac{p}{\rho},
\end{equation}
where $\gamma>1$ is the adiabatic exponent, appeared in the state function of polytropic gases:  $$p=\kappa \rho^{\gamma}\exp{(\frac{\hat{S}}{c_{\upsilon}})},$$
with $\hat{S}$ the entropy, and  $\kappa, c_{\upsilon}$ being positive constants. However, we emphasize that in this work  we shall use
\begin{equation}\label{1.3}
p=\dfrac{\gamma-1}{\gamma}\rho(E-\dfrac{1}{2}(u^{2}+v^{2})),
\end{equation}
which is solved from \eqref{1.2},
as the state function of the gas. It includes polytropic gases ($\gamma>1$) and pressureless gas ($\gamma=1$).

It is well-known that to study weak solutions, one shall choose the correct representations of the Euler equations and the state function. Previous works \cite{Qu2018Hypersonic, Qu2018High, Qu2018measure} have shown that \eqref{1.1} and \eqref{1.3} are the proper starting point to study physical problems of limiting hypersonic flows and general Radon measure solutions of compressible Euler equations.

\subsubsection{Problem 1}\label{sec111}
Consider the problem of supersonic flow passing  an infinite solid  ramp $\{(x,y)\in \Bbb R^{2}:~ x\in\mathbb{R}, \ y\leq b(x)\}$, where $b(x)$ is a given continuous function, satisfying $b(x)=0$ for $x\le0$, $b(x)\in C^{2}$ and $b'(x)\geq 0$ for $x>0$ (see Figure \ref{fig1}). It is a classic problem in gas dynamics, and has been studied extensively (see, for example, \cite{Chen2006Existence,Hu2018The,Hu2017Global} and references therein). Actually it's Hu and Zhang's work \cite{Hu2018The,Hu2017Global} that motivates us to study the hypersonic-limit flow.

To  formulate the problem, denote the region filled with gas by
$$
\Omega\triangleq\{(x,y)\in \Bbb R^{2}:~x>0,\ y>b(x)\}.
$$
The surface of the ramp is then
$$W\triangleq\{(x,y)\in \Bbb R^{2}:~x\geq 0,\ y=b(x)\},$$
on which we propose the slip condition
\begin{equation}\label{1.5}
\begin{aligned}
&v=b'(x) u   \ \ \  \text{on}\  W. \\
\end{aligned}
\end{equation}

Without loss of generality ({cf.} \cite{Qu2018Hypersonic} for some non-dimensional  scalings), we may assume that the uniform upcoming supersonic flow is
\begin{equation}\label{1.6}
\begin{aligned}
&U=U_{0}\triangleq(\rho_0,u_0,v_0,E_0)^\top=(1,1,0,E_{0})^{\top} \ \   \text{on} \ x=0, \  y>0, \\
\end{aligned}
\end{equation}
where $E_{0}>{1}/{2}$ is a given constant.
From \cite{Qu2018Hypersonic}, we know that the Mach number of the upcoming flow $M_{0}=\infty$ equals that $\gamma=1$, hence by \eqref{1.3} one has $p_{0}=0$. Therefore limiting hypersonic flow is pressureless Euler flow if there is no physical boundary.

\smallskip

{\bf Problem~1~~} Find a solution to the initial-boundary value problem \eqref{1.1}, \eqref{1.3}-\eqref{1.6} in $\Omega$ when $\gamma=1$.

\vspace{-3cm}
\begin{figure}[H]
\centering
\includegraphics[width=6in]{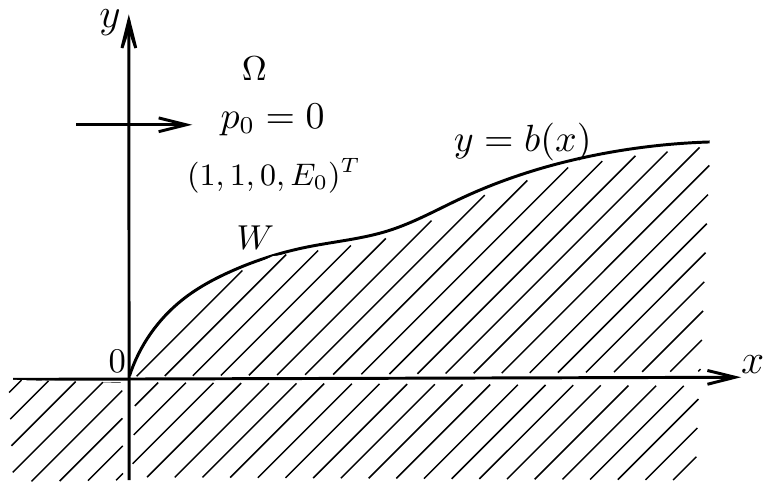}\\
\vspace{-14cm} \caption{Limiting hypersonic flow passing an infinite ramp.}\label{fig1}
\end{figure}

\subsubsection{Problem 2}\label{sec112}
For limiting hypersonic flow passing a finite ramp,
we consider the case that there is a cliff $W_{e}\triangleq\{(x,y)\in \Bbb R^{2}:~x=x_{*},\ y<b(x_{*})\}$, where $x_*>0$ is given, and there is static uniform gas near the cliff:
\begin{equation}\label{1.7}
\underline{U}\triangleq(\underline{\rho},0,0,\underline{E})^{\top},
\end{equation}
where $\underline{\rho}>0, \ \underline{E}>0$ are given constants. From \eqref{1.3} we then solve the pressure $\underline{p}\geq0$. In particular for $\underline{p}=0$ we have pressureless static flow. This problem was proposed in \cite[p.148]{Hayes2004hypersonic} and discussed in a sketchy and qualitative style there.

For limiting hypersonic flow passing the cliff, we assume that there will appear a curve, called {\it free layer} in \cite{Hayes2004hypersonic}, to separate the limiting hypersonic flow above it from the static gas below it. Let the free layer be $$S\triangleq\{(x,y)\in \Bbb R^{2}:~ x\geq x_{*},\ y=s(x)\},$$
where $y=s(x)$ is a function to be solved. Set $\Omega_{2}$ be the region bounded by $W_e$ and $S$, which is the region occupied by the uniform static gas (see Figure \ref{fig2}). For convenience, we also define $\Omega_1$ to be the region bounded by the positive $y$-axis and $W_f\cup S$, where $$W_{f}\triangleq\{(x,y)\in \Bbb R^{2}:~0\leq x\leq x_{*},\ y=b(x)\}$$ is the surface of the finite ramp.

We still propose the slip condition:
\begin{equation}\label{1.8}
\begin{aligned}
&v=b'(x) u   \ \ \  \text{on}\  W_{f}.
\end{aligned}
\end{equation}

\begin{remark}
For convenience of later reference, we emphasize that on the free layer, it also holds the slip condition
\begin{equation}\label{1.9}
\begin{aligned}
&v=s'(x) u   \ \ \  \text{on}\  S,
\end{aligned}
\end{equation}
which can be derived naturally from the definition of measure solutions that will be given below. Note that $(u,v)$ in \eqref{1.9} is the velocity of concentrated particles on the free layer, which is usually different from the velocity of the gas that is  close to the free layer.  See Remark \ref{rm31}.
\end{remark}

\begin{figure}
\centering
\includegraphics[width=6in]{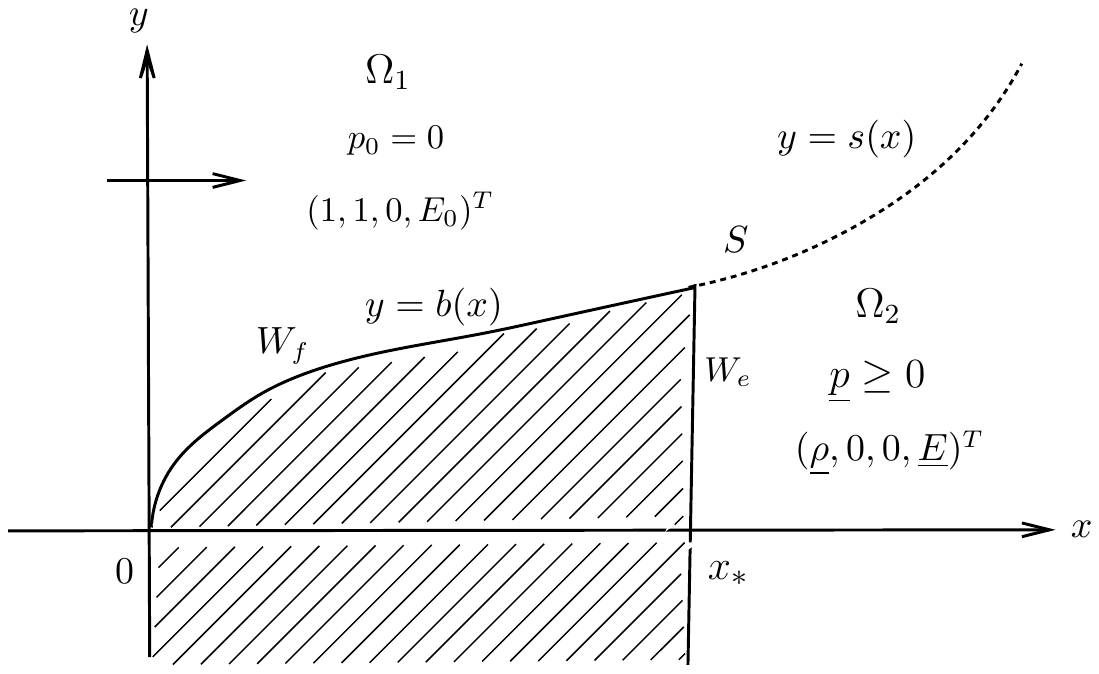}\\
\vspace{-13cm} \caption{Interaction of limiting hypersonic flow in $\Omega_1$ with uniform static gas in $\Omega_2$.}\label{fig2}
\end{figure}

{\bf Problem~2~~} Find a solution to the transonic three-phase flow \footnote{Three phases are limiting hypersonic flow, static (polytropic) gas, and the free layer itself. For the limiting hypersonic flow, the Euler equations are hyperbolic, while for polytropic static gas, it is subsonic and the governing Euler equations are of degenerate elliptic-hyperbolic composite type. Hence the whole flow field is transonic.} problem \eqref{1.4}, \eqref{1.3}, \eqref{1.6}-\eqref{1.8}.

\smallskip
We remark that although this problem looks quite similar to those studied in \cite{Chen2013Stability,Chen2013Well} for supersonic polytropic gas flow passing over a ``dead gas zone", the results and methods are quite different. For the latter there is a classical contact discontinuity to separate the moving gas and static gas, and it cannot bear any pressure difference.  This is an example that displays the difficulty and fascination of studying the compressible Euler equations: {\it Similarly-looking problems may have drastic differences inside.}

\subsubsection{Problems 3}

Considering applications to rocket engineering, suppose now that $W_f$ is a wall of a two-dimensional nozzle, and $W_e$ is the exit of the nozzle, where the gas flows out (i.e., jet) is assumed to be uniform, pressureless and hyperbolic (see Figure \ref{fig3}):
\begin{equation}\label{1.10}
\begin{aligned}
\underline{U}=(\underline{\rho},\underline{u},\underline{v},\underline{E})^{\top} \qquad \text{on}\ \ W_e,
\end{aligned}
\end{equation}
where $\underline{\rho}>0,\ \underline{u}>0,\  \underline{E}>0$ and $\underline{v}$ are given constants.
As before, we may assume that there is a free layer to separate the limiting hypersonic flow above it and the pressureless jet below it.

{\bf Problem~3~~} Find a solution to the  problem {\eqref{1.4}}, \eqref{1.3}, \eqref{1.6}, \eqref{1.8}, \eqref{1.10}.

\begin{remark}
The steady pressureless Euler system is hyperbolic in the positive $x$-direction if and only if $u>0$, with $v/u$ being its eigenvalues, cf. \cite[(2.1) in p.325]{Cheng2011Delta}.
\end{remark}

\begin{figure}
\centering
\includegraphics[width=6in]{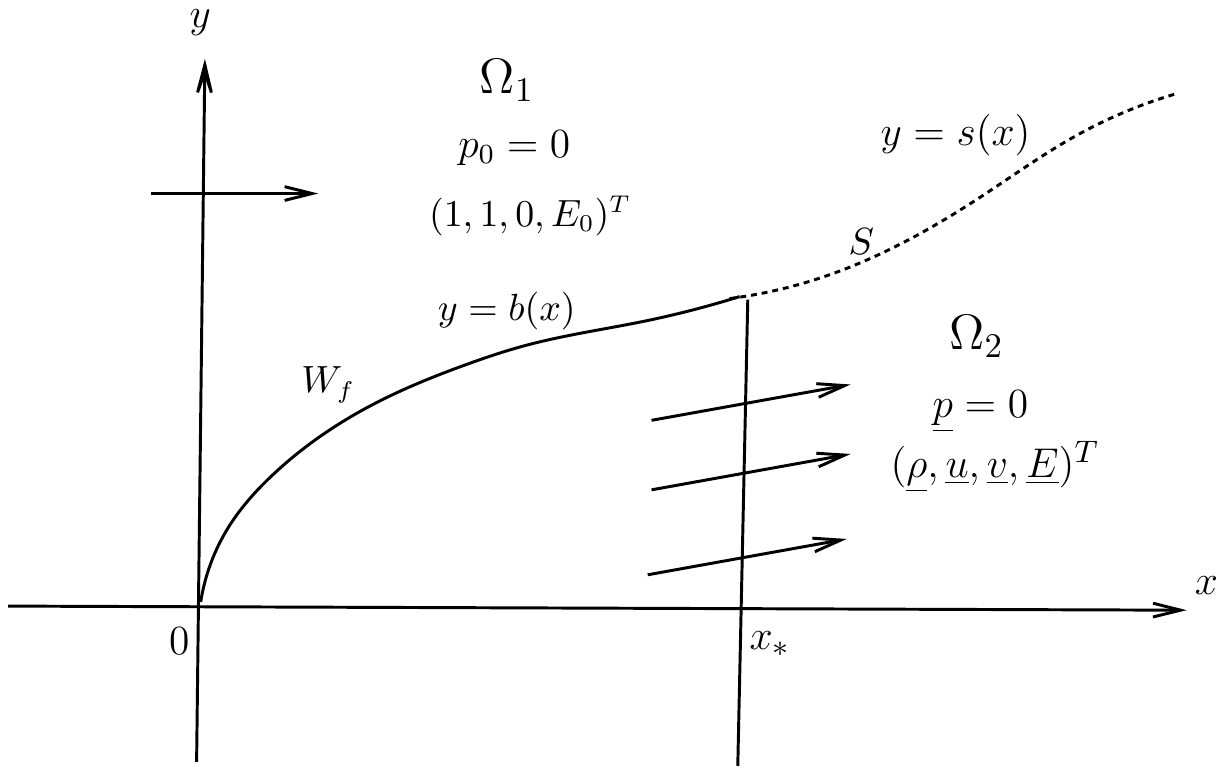}\\
\vspace{-12.5cm} \caption{Interactions of limiting hypersonic flows and pressureless jets.}\label{fig3}
\end{figure}

\subsection{Definition of Radon measure solutions}\label{sec12}
We now clarify the meaning of solutions to the above problems. As we know from physical observations, there appear concentration of mass along walls in these problems, and one needs Dirac measures supported on curves to describe such phenomena. So the key point is how to understand the Euler equations when some of the unknowns are measures.

We review some basic notions of measure theory. Let $\mathcal{B}$ be the Borel $\sigma$-algebra on $\Bbb R^{2}$, and $m$ a  Radon measure on $\mathcal{B}$.
As a Radon measure, $m$ can be considered as a bounded linear functional on the space $C_{0}(\Bbb R^{2})$ which consists  of continuous function with compact support:  for a test function $\phi \in C_{0}(\Bbb R^{2})$, one has
\begin{equation}\label{2.1}
\langle m,\phi\rangle=\int_{\Bbb R^{2}}{\phi(x,y)m(\dd x\dd y)}.
\end{equation}
For example, $w_{L}\delta_{L}$, a Dirac measure supported on a Lipschitz curve $L$ with weight $w_L$, is defined by
\begin{equation}\label{2.2}
\langle w_{L}\delta_{L},\phi\rangle=\int_{0}^{T}{w_{L}(t)\phi(x(t),y(t))
\sqrt{x'(t)^{2}+y'(t)^{2}}}\,\dd t,   \ \ \  \forall \phi\in C_{0}(\Bbb R^{2}).
\end{equation}	
Here $L=\{(x(t),y(t)):~t\in [0,T)\}$ is a Lipschitz curve given by the parameter $t$, and $w_{L}(t)\in L^{1}_{\mathrm{loc}}(0,T)$. It is singular to the Lebesgue measure  $\mathcal{L}^{2}$ on the plane.  The standard notation $\lambda\ll\mu$ means that a measure $\lambda$ is absolutely continuous with respect to a nonnegative measure $\mu$ (cf. \cite[p.50]{Evans2015}).

We now propose a rigorous definition of Radon measure solutions to Problem 1.

\begin{definition}\label{def2.2}
For fixed $\gamma\geq 1$,  let $m^{i},n^{i} \ (i=0,1,2,3),\wp$ be Radon measures on $\overline{\Omega}$, and $w_{p}$ a nonnegative function belong to  $L^{1}_{\mathrm{loc}}(\Bbb R^{+}\cup \{0\})$. We call $(\varrho,u,v,E)$ a (Radon) measure solution to Problem 1, provided that the following are valid:
\par i) For any $\phi\in C_{0}^{1}(\Bbb R^{2})$ (continuously differentiable functions with compact supports), there hold
\begin{equation}\label{2.3}
\langle m^{0},\partial_{x}\phi\rangle+\langle n^{0},\partial_{y}\phi\rangle+\int_{0}^{\infty}\rho_{0}u_{0}\phi(0,y)\,\dd y=0,	
\end{equation}	
\begin{equation}\label{2.4}
\begin{aligned}
&\langle m^{1},\partial_{x}\phi\rangle+\langle n^{1},\partial_{y}\phi\rangle+\langle\wp,\partial_{x}\phi\rangle+\langle w_{p}n_{1}\delta_{W},\phi\rangle\\&\qquad+\int_{0}^{\infty}(\rho_{0}u_{0}^{2}+p_{0})\phi(0,y)\,\dd y=0,
\end{aligned}
\end{equation}
\begin{equation}\label{2.5}
\begin{aligned}
&\langle m^{2},\partial_{x}\phi\rangle+\langle n^{2},\partial_{y}\phi\rangle+\langle\wp,\partial_{y}\phi\rangle+\langle w_{p}n_{2}\delta_{W},\phi\rangle\\&\qquad+\int_{0}^{\infty}(\rho_{0}u_{0}v_{0})\phi(0,y)\,\dd y=0,
\end{aligned}
\end{equation}
\begin{equation}\label{2.6}
\begin{aligned}
&\langle m^{3},\partial_{x}\phi\rangle+\langle n^{3},\partial_{y}\phi\rangle+\int_{0}^{\infty}(\rho_{0}u_{0}E_{0})\phi(0,y)\,\dd y=0,
\end{aligned}
\end{equation}	
where $\mathbf{n}=(n_{1}, n_{2})=(-b'(x),1)/\sqrt{1+b'(x)^{2}}$ is the inner unit normal vector on $W$ (pointing into $\Omega$);
\par ii) $\varrho$ is a nonnegative Radon measure,
such that $\wp\ll\varrho$, $(m^{0},n^{0})\ll\varrho$, $(m^{k},n^{k})\ll(m^{0},n^{0})\ (k=1,2,3)$,
and the corresponding Radon-Nikodym derivatives satisfy $\varrho-a.e.$ that
\begin{equation}\label{2.7}
\begin{aligned}
&u=\dfrac{m^{0}(\dd x\dd y)}{\varrho(\dd x \dd y)} \ \ \ \text{and} \ \ \   v=\dfrac{n^{0}(\dd x\dd y)}{\varrho(\dd x \dd y)},
\end{aligned}
\end{equation}
\begin{equation}\label{2.8}
\begin{aligned}
&u=\dfrac{m^{1}(\dd x\dd y)}{m^{0}(\dd x \dd y)}=\dfrac{n^{1}(\dd x \dd y)}{n^{0}(\dd x \dd y)}, \\  &v=\dfrac{m^{2}(dxdy)}{m^{0}(\dd x \dd y)}=\dfrac{n^{2}(\dd x \dd y)}{n^{0}(\dd x \dd y)},
\end{aligned}
\end{equation}
and there is a $\varrho$-a.e. function $E$, so that
\begin{equation}\label{2.9}
\begin{aligned}
&E=\dfrac{m^{3}(\dd x \dd y)}{m^{0}(\dd x \dd y)}=\dfrac{n^{3}(\dd x \dd y)}{n^{0}(\dd x \dd y)};
\end{aligned}
\end{equation}
on the null sets of $\varrho$, we set $u, v, E$ to be zero;
\par iii) If  $\varrho\ll\mathcal{L}^{2}$ in a set $A$, with $\rho$ being the Radon-Nikodym derivative, and $\wp\ll\mathcal{L}^{2}$ in $A$ with Radon-Nikodym derivative $p$, then
\eqref{1.3} is valid $\mathcal{L}^{2}$-a.e. in $A$, and the classical entropy conditions hold across discontinuities of the functions $U=(\rho, u,v, E)^\top$ in $A$.
\end{definition}

\begin{remark}\label{rm11} It is easy to check that integral weak solutions are measure solutions, {cf.} \cite{Qu2018Hypersonic}. In that paper, it is also shown that for a straight ramp, the well-known piecewise constant integral weak solution containing a shock converges weakly in the sense of measure to the corresponding measure solution with density containing a Dirac measure on the ramp, as the Mach number of the upcoming flow goes to infinity (or equivalently, $\gamma\to1$). Hence  consistency  holds for the above definition of measure solutions.
\end{remark}

\begin{remark}\label{rm12}
The basic idea of our definition of  measure solution is firstly to relax the Euler equations to a linear differential system of measures of fluxes of mass and momentum etc., and the nonlinearity of the Euler system is recovered from the algebraic relations of Radon-Nikodym derivatives. The state function \eqref{1.3} is no longer required when concentration occur. So there will not appear difficulty such as  products of Dirac measures. The definition makes sense for general multidimensional steady or unsteady compressible Euler equations.
\end{remark}

\begin{remark}\label{rm13}
$w_p \cdot(-\mathbf{n})$ is the force (lift/drag) acting on the ramp by the gas flow, hence it is a quantity received great attention from engineers. We require in this paper $w_p>0$  to guarantee that the mass concentrates actually on the walls.
\end{remark}

\begin{remark}\label{rm14}
We can define measure solutions to Problems 2 and 3 in a similar way; consult necessary modifications indicated in Sections \ref{sec3}-\ref{sec4}.
\end{remark}

\subsection{Main results and remarks} \label{sec13}
The following are the  main results we obtained for the above three problems.

\begin{theorem}\label{thm1.1}
For limiting hypersonic flow passing the infinite ramp $W$,
suppose that
\begin{equation}\label{120}
b'(x)\geq0,\quad  b''(x){H(x)}>-{b'(x)^{2}\sqrt{1+b'(x)^{2}}},
\end{equation}
where
\begin{equation} \label{121} H(x)\triangleq\int_{0}^{x}\dfrac{b'(t)}{\sqrt{1+b'(t)^{2}}}\,\dd t.
\end{equation}
Then Problem 1 has a measure solution given by \eqref{2.31}\eqref{2.30}\eqref{222}, with density containing a weighted Dirac measure supported on $W$. In particular, we have
\begin{equation}\label{122}
w_{p}(x)=\dfrac{b''(x)H(x)+b'(x)^{2}\sqrt{1+b'(x)^{2}}}{(1+b'(x)^{2})^{\frac{3}{2}}}.
\end{equation}

\end{theorem}
\begin{remark}\label{rm15}
Formula \eqref{122} is the celebrated {\it Newton-Busemann pressure law.} As pointed out in \cite[p.133]{Hayes2004hypersonic}: ``This formula is valuable because it is easy to compute and gives a simple basis of comparison. " Adolf Busemann (1901--1986) was a German aerospace engineer who also discovered the benefits of the swept wing for aircraft at high speeds.

A derivation of the Newton-Busemann pressure law is given in \cite[Sections 3.3, 3.4]{Anderson2006hypersonic}, and the law is presented as formula (3.29) in \cite[p.67]{Anderson2006hypersonic}. The formula is derived by a lengthy physical argument (nearly four pages), taking into account the centrifugal force required for a particle to moving along the curved ramp. Hence in \cite[p.133]{Hayes2004hypersonic} Hayes and Probstein wrote ``It is not based on any rational theory, however, and its empirical basis should be kept in mind." We believe a significant contribution of this paper is that it establishes a rigorous mathematical foundation for the  Newton-Busemann pressure law, as the law can now be proved by short and straightforward computations from the very fundamental compressible Euler equations. This makes the Newton theory a part of modern rational mechanics.

To show equivalence of \eqref{122} with (3.29) in  \cite[p.67]{Anderson2006hypersonic}, note that $b'(x)=\tan\theta$ for $\theta$ appeared in (3.29) in  \cite[p.67]{Anderson2006hypersonic}, and then rewrite the integration in \eqref{122} to be integrated for the variable $y$, using $\dd y=b'(x)\,\dd x$. The appearance of the factor $2$ in (3.29) in  \cite[p.67]{Anderson2006hypersonic} is that Anderson employed the scaling (3.16) in \cite[p.61]{Anderson2006hypersonic} to define the pressure coefficients $C_p$, where there is a $\frac12$ in the denominator; while in our work we just used the scaling $\tilde{p}=\frac{p}{\rho_\infty u_\infty^2}$ without the factor $\frac12$ in the denominator (see the scalings below (2.6) in \cite{Qu2018Hypersonic}).
\end{remark}

\begin{remark}
Obviously, \eqref{120} is sufficient to guarantee that $w_p>0,$  which means the ramp suffers force from the flow at each point. Such nontrivial ramp exits. For example, taking $b(x)=\sqrt{x}$, direct computation shows that  \eqref{120} holds for all $ x\in[0, \infty)$.
\end{remark}

\begin{theorem}\label{thm1.2}
For limiting hypersonic flow passing a finite ramp $W_{f}$, suppose \eqref{120} is valid for $0\leq x\le x_{*}$ and $H(x_*)>0$. Then for $\underline{p}\in[0,1]$, Problem 2 has a global measure solution (see \eqref{3.14}\eqref{3.15}) defined for all $x\ge0$, with density containing a weighted Dirac measure supported on a curve which coincides with $W_f$ for $0\leq x\le x_{*}$.  The curve for $x>x_*$, called ``free layer", separates limiting hypersonic flow above it from the static gas below it.  The shape of the free layer depends on the pressure of the static gas ({\it cf.} Figure \ref{fig5} and Figure \ref{fig8}):

\par  1) If $\underline{p}=0$, the free layer (see \eqref{3.31}) is at most of the order $\sqrt{x}$  as $x\to\infty$;

\par 2) If $0<\underline{p}<1$ , the free layer (see \eqref{3.18}) is of the order $\sqrt{\dfrac{\underline{p}}{1-\underline{p}}}x$ as $x\to\infty$;

\par 3) If $\underline{p}=1$, the free layer (see \eqref{3.32}) is of the order  $x^{2}$ as $x\to\infty$.

If $\underline{p}>1$,  there is a finite point $x_{\vartriangle}>x_*$, and Problem 2 has a local measure solution (see \eqref{3.14}\eqref{3.15}) with the above structure defined only on $x\in[0, x_{\vartriangle}]$. The solution blows up at $x=x_{\vartriangle}$, in the sense that the free layer (see \eqref{331}) rolls up at $x=x_{\vartriangle}$  and cannot be prolonged.
\end{theorem}

\begin{remark}\label{rm18}
According to \cite[p.144]{Hayes2004hypersonic}, the concept of free layer was introduced by Busemann in 1933, to indicate a shock layer the pressure behind which is zero. We use free layer in this paper in a more general sense, which is also called {\it delta shock} in literature of mathematics (see, for example,  \cite{Chen2003Formation, Cheng2011Delta, shensun2009, Tan1994Delta}). There were many discussions and conjectures on free layers in \cite[Section 3.3]{Hayes2004hypersonic}. The merit of our approach is that we can calculate explicitly the expressions of various free layers based on rigorous mathematics. An application of designing afterbody that bears no force in limiting hypersonic flow is discussed in Remark \ref{rm33}.
\end{remark}

\begin{remark}\label{rm19}
To our knowledge, the last conclusion in the above theorem presents  the first example of blowing up of measure solutions (delta shocks) for the  Euler equations.
\end{remark}

\begin{theorem}\label{thm1.3}
Under the same assumptions on $b(x)$ as in Theorem \ref{thm1.2}, for Problem  3, i.e., interactions of limiting hypersonic flows and pressureless jets, we have the following results:

1) If $\underline{v}/\underline{u}\ge b'(x_*)$, then Problem 3 has a measure solution (see \eqref{4.16}-\eqref{4.17}), containing a free layer (see \eqref{422} and \eqref{4.19}) separating the limiting hypersonic flow and pressureless jet, which is of the order $\displaystyle \frac{\sqrt{\underline{\rho}}\underline{v}}{1+\sqrt{\underline{\rho}}\underline{u}}x$ as $x\to\infty$.

2) If $\underline{v}/\underline{u}<b'(x_*),$  then Problem 3 has a global measure solution (see \eqref{443}\eqref{444}) containing vacuum. The vacuum starts at $(x_*,b(x_*))$, and is bounded by a free layer (see \eqref{4.39}) and a straight contact discontinuity (see \eqref{4.42}). Furthermore,  if $\underline{v}\leq0$, the vacuum is unbounded. If $\underline{v}>0$, the vacuum is bounded.
\end{theorem}

\begin{remark}
To exclude the possibility that particles escape from the free layer and hence leads to obvious non-uniqueness of measure solutions, we used here the well-recognized  entropy condition of delta shocks \eqref{45}. The free layers obtained for item 1) in the above theorem satisfy this entropy condition.
\end{remark}

\begin{remark}
To prove item 2) in the theorem, we encounter a problem of colliding of free layer and contact discontinuity, and it is reduced to the case studied in 1).
\end{remark}

We will prove Theorems \ref{thm1.1}, \ref{thm1.2} and \ref{thm1.3}  in Sections \ref{sec2}-\ref{sec4} respectively. In a short Section \ref{sec5}, we focus on the role played by singular Riemann problems in the studies of general measure solutions.  Here by singular Riemann problem we mean the initial data is piecewise constant, with density containing  a Dirac measure supported on the initial discontinuity point.

Finally we review briefly some mathematical studies on Radon measure solutions of hyperbolic equations. For scalar conservation laws, Liu and Pierre \cite{LP1984} had already found regularizing effects of genuinely nonlinear fluxes --- although the initial data could contain Dirac measures, the solutions are always  functions. Demengel and Serre \cite{DS1991} studied well-posedness of Cauchy problems of scalar conservation laws with general convex fluxes that grow linearly at infinity, and the initial data being non-negative measures. The main tools are Lax-Oleinik formula and theory of Hamilton-Jacobi equations. See \cite{BSTT2020} for recent developments. Dal Maso, LeFloch and Murat \cite{DL1995} introduced a product of a measure and certain discontinuous function, and used it to define and study  measure solutions to some $2\times 2$ hyperbolic system \cite{L1990, HL1996}. There are also many works studying measure solutions using various flux approximation/regularization, such as vanishing viscosity \cite{shengzhang1999, Tan1994Delta} or vanishing pressure \cite{Chen2003Formation,CHEN2004Concentration}. Huang and Wang established  well-posedness in the class of Radon measures for Cauchy problem of the one-dimensional pressureless Euler equations \cite{Huang2001Well}. See \cite{Cavalletti2019A} for recent progress on the multidimensional case. We recommend \cite{YZ} for a rather complete survey of mathematical studies of delta shocks.

Comparing to these established works, the merit of our approach is that our definition of Radon measure solution is rather elementary and flexible, applicable to a large extent of problems, and we can prove from it naturally some physically well-known formulas. However, since the definition  employed the special structure of compressible Euler equations,  we do not know presently how to extend it to general hyperbolic systems of conservation laws.

\section {Limiting hypersonic flow passing an infinite ramp and Newton-Busemann pressure law} \label{sec2}

In this section we prove Theorem \ref{thm1.1} by constructing a measure solution to Problem 1.

Let $\mathsf{I}_{A}$ be the characteristic function of a set $A\subset\mathbb{R}^2$, namely $\mathsf{I}_{A}(x,y)=1$ if $(x,y)\in A$ and $\mathsf{I}_{A}(x,y)=0$ otherwise. Recall that $\LL^2$ is the standard Lebesgue measure on the plane $\mathbb{R}^2$. Now suppose the measures of fluxes are given by
\begin{eqnarray}
&&\begin{aligned}
      &m^{0}=\rho_{0}u_{0}\mathsf{I}_{\Omega}\mathcal{L}^{2}+w_{m}^{0}(x)\delta_{W} =\mathsf{I}_{\Omega}\mathcal{L}^{2}+w_{m}^{0}(x)\delta_{W},\\ &n^{0}=\rho_{0}v_{0}\mathsf{I}_{\Omega}\mathcal{L}^{2}+w_{n}^{0}(x)\delta_{W}=w_{n}^{0}(x)\delta_{W};
  \end{aligned}\label{2.10}\\
&&m^{1}=\mathsf{I}_{\Omega}\mathcal{L}^{2}+w_{m}^{1}(x)\delta_{W}, \quad n^{1}=w_{n}^{1}(x)\delta_{W},  \ \  \wp=0;\label{2.06}\\
&&m^{2}=w_{m}^{2}(x)\delta_{W},  \ \ \ \  n^{2}=w_{n}^{2}(x)\delta_{W};\label{2.20}\\
&&m^{3}=E_{0}\mathsf{I}_{\Omega}\mathcal{L}^{2}+w_{m}^{3}(x)\delta_{W}, \quad n^{3}=w_{n}^{3}(x)\delta_{W},\label{2.05}
\end{eqnarray}
where $w_{m}^{i}(x),w_{n}^{i}(x)\ (i=0,1,2,3)$ are functions to be determined.

\medskip

Substituting \eqref{2.10} into \eqref{2.3}, it follows
\begin{equation}\label{2.11}
\begin{aligned}
\int_{\Omega}{\partial_{x}\phi(x,y)} \,\dd x \dd y &+\int_{0}^{\infty}w_{m}^{0}(x)\partial_{x}\phi(x,b(x))\sqrt{1+b'(x)^{2}}\,\dd x\\
&+\int_{0}^{\infty}w_{n}^{0}(x)\partial_{y}\phi(x,b(x))\sqrt{1+b'(x)^{2}}\,\dd x+\int_{0}^{\infty}\phi(0,y)\,\dd y=0.
\end{aligned}
\end{equation}
Observing that
\begin{equation*}
  \begin{aligned}
     &\int_{0}^{\infty}w_{m}^{0}(x)\partial_{x}\phi(x,b(x))\sqrt{1+b'(x)^{2}}\,\dd x \\ =&-{w_m^0(0)}\sqrt{1+{b'(0)^2}}\phi(0,0)-\int_{0}^{\infty} b'(x)w_{m}^{0}(x)\sqrt{1+b'(x)^{2}}\partial_{y}\phi(x,b(x))\,\dd x \\ &-\int_{0}^{\infty}\dfrac{\dd(w_{m}^{0}(x)\sqrt{1+b'(x)^{2}})}{\dd x}\phi(x,b(x))\,\dd x,
  \end{aligned}
\end{equation*}
we have
\begin{equation}\label{2.12}
 \begin{aligned}
 &w_{m}^{0}(0)\sqrt{1+b'(0)^{2}}\phi(0,0)-\int_{0}^{\infty}b'(x)\phi(x,b(x))\,\dd x +\int_{0}^{\infty}\dfrac{\dd(w_{m}^{0}(x)\sqrt{1+b'(x)^{2}})}{\dd x}\phi(x,b(x))\,\dd x\\
 &-\int_{0}^{\infty}(w_{n}^{0}(x) -b'(x)w_{m}^{0}(x))\sqrt{1+b'(x)^{2}}\partial_{y}\phi(x,b(x))\,\dd x=0.
 \end{aligned}
 \end{equation}
By arbitrariness of $\phi$, the above implies
 \begin{equation}\label{2.13}
 \begin{aligned}
 &w_{m}^{0}(0)=0, \ \
 \dfrac{\dd(w_{m}^{0}(x)\sqrt{1+b'(x)^{2}})}{\dd x}=b'(x), \  \ w_{n}^{0}(x)=b'(x)w_{m}^{0}(x),
 \end{aligned}
 \end{equation}
and we solve
   \begin{equation}\label{2.14}
  \begin{aligned}
  &w_{m}^{0}(x)=\dfrac{b(x)}{\sqrt{1+b'(x)^{2}}}, \ \ \  w_{n}^{0}(x)=\dfrac{b'(x)b(x)}{\sqrt{1+b'(x)^{2}}}.
  \end{aligned}
  \end{equation}
We may get
\begin{equation}\label{2.15}
w_{m}^{3}(x)=\dfrac{E_{0}b(x)}{\sqrt{1+b'(x)^{2}}}, \quad w_n^{3}(x)=\dfrac{E_{0}b'(x)b(x)}{\sqrt{1+b'(x)^{2}}}
\end{equation}
in the same way. \footnote{Actually for this problem there is a freedom to determine the value of $E$ in these two weights. We choose here $E_0$ on $W$, because for supersonic flows without concentration, it is well-known that the quantity $E$ is always constant along flow trajectories. So for this problem we also take it as constant in the whole flow field due to our uniform initial data.}

Similarly substituting \eqref{2.06} into \eqref{2.4}, we have
 \begin{equation}\label{2.17}
\begin{aligned}
&\int_{0}^{\infty}[{(1-w_{p}(x))}b'(x)-\dfrac{\dd(w_{m}^{1}(x)\sqrt{1+b'(x)^{2}})}{\dd x}
]\phi(x,b(x))\,\dd x-w_{m}^{1}(0)\sqrt{1+b'(0)^{2}}\phi(0,0)\\&+\int_{0}^{\infty}(w_{n}^{1}(x)-b'(x)w_{m}^{1}(x))\sqrt{1+b'(x)^{2}}\partial_{y}\phi(x,b(x))\,\dd x=0,
\end{aligned}
\end{equation}
which yields
 \begin{eqnarray}\label{2.18}
&&w_{m}^{1}(0)=0, \ \  w_{n}^{1}(0)=0, \ \  w_{n}^{1}(x)=b'(x)w_{m}^{1}(x),\\
&&
\dfrac{\dd(w_{m}^{1}(x)\sqrt{1+b'(x)^{2}})}{\dd x}={(1-w_{p}(x))}b'(x),\qquad x\ge0.\label{2.19}
\end{eqnarray}
Note the function $w_p(x)$ shall be solved.

Thanks to \eqref{2.20}, and noticing that $v_{0}=0$,  \eqref{2.5} becomes
 \begin{equation}
\begin{aligned}
&-w_{m}^{2}(0)\sqrt{1+b'(0)^{2}}\phi(0,0)+\int_{0}^{\infty}[w_{p}(x)-
\dfrac{\dd(w_{m}^{2}(x)\sqrt{1+b'(x)^{2}})}{\dd x}]
\phi(x,b(x))\dd x\\
&+\int_{0}^{\infty}(w_{n}^{2}(x) -b'(x)w_{m}^{2}(x))\sqrt{1+b'(x)^{2}}\partial_{y}\phi(x,b(x))\dd x=0,
\end{aligned}
\end{equation}
consequently
\begin{eqnarray}
&&w_{m}^{2}(0)=0, \ \  w_{n}^{2}(x)=b'(x)w_{m}^{2}(x),\label{2.22}\\
&&\dfrac{\dd(w_{m}^{2}(x)\sqrt{1+b'(x)^{2}})}{\dd x}={w_{p}(x)},\quad x\ge0.\label{2.23}
\end{eqnarray}

By requirements \eqref{1.5} and \eqref{2.8}, there holds
\begin{equation}\label{2.24}
\begin{aligned}
&w_{m}^{2}(x)=b'(x)w_{m}^{1}(x).
\end{aligned}
\end{equation}
Then we solve from \eqref{2.19}, \eqref{2.23} and \eqref{2.24} that
\begin{align}\label{2.25}
&w_{m}^{1}(x)=\dfrac{H(x)}{1+b'(x)^{2}},\\
&w_{p}(x)=\dfrac{b''(x)H(x)+b'(x)^{2}\sqrt{1+b'(x)^{2}}}{(1+b'(x)^{2})^{\frac{3}{2}}},\label{2.26}
\end{align}
where
\begin{equation}\label{219}
H(x)\triangleq\int_{0}^{x}\dfrac{b'(t)}{\sqrt{1+b'(t)^{2}}}\,\dd t.\end{equation}
We thus proved the Newton-Busemann pressure law \eqref{122}.

To write down a measure solution, applying \eqref{2.8}, one has
\begin{equation}
\begin{aligned}
&u|_{W}=\dfrac{H(x)}{b(x)\sqrt{1+b'(x)^{2}}}, \ \ \  v|_{W}=\dfrac{b'(x)H(x)}{b(x)\sqrt{1+b'(x)^{2}}},
\end{aligned}
\end{equation}
hence
\begin{equation}\label{2.30}
\begin{aligned}
&u=\mathsf{I}_{\Omega}+\dfrac{H(x)}{b(x)\sqrt{1+b'(x)^{2}}}\mathsf{I}_{W}, \ \ \  v=\dfrac{b'(x)H(x)}{b(x)\sqrt{1+b'(x)^{2}}}\mathsf{I}_{W}.
\end{aligned}
\end{equation}
By \eqref{2.7}, one gets the measure of density:
\begin{equation}\label{2.31}
\begin{aligned}
&\varrho=\mathsf{I}_{\Omega}\mathcal{L}^{2}+\dfrac{(b(x))^{2}}{H(x)}\mathsf{\delta}_{W}.
\end{aligned}
\end{equation}
Furthermore, recalling that \eqref{2.9}, \eqref{2.14} and \eqref{2.15}, then
\begin{equation}\label{222}
E=E_{0}\mathsf{I}_{\Omega}+E_{0}\mathsf{I}_{W}.
\end{equation}
So \eqref{2.30}, \eqref{2.31} and \eqref{222} is a measure solution to Problem 1.
This completes the proof of Theorem \ref{thm1.1}.

\begin{remark}\label{rm21}
The reason why we require $w_p>0$ in Theorem \ref{thm1.1} is to guarantee the assumption lying in \eqref{2.10}-\eqref{2.05}, namely concentration of mass appears just on the surface of the ramp. If at some point $w_p=0$, then the particles in the shock layer may not feel the ramp and then fly away from the ramp, hence a free layer appears. The region between the free layer and the ramp is vacuum, or with zero pressure, a situation that we will discuss in the next section.
\end{remark}

\begin{remark}\label{rm22}
One may wonder if there is a ramp that bears uniform force $p$ per unit area from the upcoming hypersonic limit flow. This means the function  $b(x)$ solves the nonlocal ordinary differential equations
\begin{equation}
\label{225}b''(x)H(x)+{b'(x)^{2}\sqrt{1+b'(x)^{2}}}=p(1+b'(x)^{2})^{3/2},
\end{equation}
where $H(x)$ is defined by \eqref{219}. Some computation reduces this equation to $H''H+(H')^2=p$, and by $H(0)=0$  we solve that $b(x)=\sqrt{\frac{p}{1-p}}x$, a case studied in \cite{Qu2018Hypersonic}. So there is no nontrivial ramp that bears no force in limiting hypersonic flow. It would be interesting to compare  this result with Remark \ref{rm33} in the following section.
\end{remark}

\section {Limiting hypersonic flow passing a finite ramp and interactions with static gas} \label{sec3}

In this section we  prove Theorem \ref{thm1.2}. To define a measure solution, recall that
the domain { we consider} now is
$$\widetilde{\Omega}\triangleq\Omega_{1}\cup\Omega_{2},$$
where
$$\Omega_{1}\triangleq\{(x,y)\in \Bbb R^{2}:~0<x\leq x_{*},\ y>b(x);\  \it x>x_{*},y>s(x)\}$$
and
$$\Omega_{2}\triangleq\{(x,y)\in \Bbb R^{2}:~ x>x_{*},\ y<s(x)\}
$$
represent respectively the region occupied by the limiting hypersonic flow above the free layer $S$ and the region behind the ramp and below $S$, see Figure \ref{fig2}.
Comparing to Problem 1, the solid boundary now is $\widehat{W}=W_f\cup W_e$. Therefore, for a definition of measure solutions of Problem 2, we just replace the boundary $W$ appeared in \eqref{2.4}  and \eqref{2.5} in Definition \ref{def2.2} by $\widehat{W}$, with
$\mathbf{n}=(n_{1}, n_{2})={(-b'(x), 1)/\sqrt{1+b'(x)^{2}}}$ being inner unit normal vector on
$W_{f}$ (pointing into $\Omega_{1}$) and
$\mathbf{n}=(n_{1}, n_{2})={(1,0)}$ on $W_{e}$, pointing into $\Omega_{2}$.

It turns out that there is a measure solution to Problem 2, which is  piecewise constant, connected by a free layer $S$. We firstly construct such a solution and then study the dependence of the shape of $S$ on the pressure $\underline{p}$ of the static gas.

\subsection{Construction of piecewise constant measure solutions}\label{sec31}

Set
$\widetilde{W}=W_{f}\cup S$, and
\begin{eqnarray}
&&
\begin{aligned}
&m^{0}=\mathsf{I}_{\Omega_{1}}\mathcal{L}^{2}+w_{m}^{0}(x)\delta_{\widetilde{W}}
=\mathsf{I}_{\Omega_{1}}\mathcal{L}^{2}+w_{m}^{0}(x)\delta_{W_{f}} +\widetilde{w_{m}^{0}}(x)\delta_{S},\\  &n^{0}=w_{n}^{0}(x)\delta_{\widetilde{W}}=w_{n}^{0}(x)\delta_{W_{f}}+\widetilde{w_{n}^{0}}(x)\delta_{S}, \end{aligned}\label{3.1}\\
&&
\begin{aligned}
&m^{1}=\mathsf{I}_{\Omega_{1}}\mathcal{L}^{2}+w_{m}^{1}(x)\delta_{\widetilde{W}}
=\mathsf{I}_{\Omega_{1}}\mathcal{L}^{2}+w_{m}^{1}(x)\delta_{W_{f}} +\widetilde{w_{m}^{1}}(x)\delta_{S},\\  &n^{1}=w_{n}^{1}(x)\delta_{\widetilde{W}}=w_{n}^{1}(x)\delta_{W_{f}} +\widetilde{w_{n}^{1}}(x)\delta_{S},
  \quad \wp=\underline{p}\mathsf{I}_{\Omega_{2}}\mathcal{L}^{2},
\end{aligned}\label{3.7}\\
&&
\begin{aligned}
&m^{2}=w_{m}^{2}(x)\delta_{\widetilde{W}}=w_{m}^{2}(x)\delta_{W_{f}}+\widetilde{w_{m}^{2}}(x)\delta_{S}, \\ &n^{2}=w_{n}^{2}(x)\delta_{\widetilde{W}}=w_{n}^{2}(x)\delta_{W_{f}}+\widetilde{w_{n}^{2}}(x)\delta_{S}, \end{aligned}\label{3.10}\\
  &&
 \begin{aligned}
 &m^{3}=E_{0}\mathsf{I}_{\Omega_{1}}\mathcal{L}^{2}+w_{m}^{3}(x)\delta_{W_{f}} +\widetilde{w_{m}^{3}}(x)\delta_{S}, \\
&n^{3}=w_{n}^{3}(x)\delta_{W_{f}}+\widetilde{w_{n}^{3}}(x)\delta_{S},
\end{aligned}
\end{eqnarray}%
where $\widetilde{w_{m}^{i}}(x),\widetilde{w_{n}^{i}}(x) \ (i=0,1,2,3)$ are new unknown weights on the free layer.

The calculations to solve the measure solution is quite similar to those presented in Section \ref{sec2}. Substituting \eqref{3.1} into \eqref{2.3}, we find
\begin{equation}
\begin{aligned}
\int_{\Omega_{1}}&{\partial_{x}\phi(x,y)} \,\dd x\dd y+\int_{0}^{x_{*}}w_{m}^{0}(x)\partial_{x}\phi(x,b(x))\sqrt{1+b'(x)^{2}}\,\dd x \\
&+\int_{x_{*}}^{\infty}\widetilde{w_{m}^{0}}(x)\partial_{x}\phi(x,s(x))\sqrt{1+s'(x)^{2}}\,\dd x+ \int_{0}^{x_{*}}w_{n}^{0}(x)\partial_{y}\phi(x,b(x))\sqrt{1+b'(x)^{2}}\,\dd x \\
&+\int_{x_{*}}^{\infty}\widetilde{w_{n}^{0}}(x)\partial_{y}\phi(x,s(x))\sqrt{1+s'(x)^{2}}\,\dd x +\int_{0}^{\infty}\phi(0,y)\,\dd y=0.
\end{aligned}
\end{equation}
Applying divergence theorem, and noticing that $s(x_{*})=b(x_{*})$, it follows
\begin{equation}\label{3.5}
\begin{aligned}
&(w_{m}^{0}(x_{*})\sqrt{1+b'(x_{*})^{2}}-\widetilde{w_{m}^{0}}(x_{*})\sqrt{1+s'(x_{*})^{2}}
)\phi(x_{*},b(x_{*}))\\
&\qquad+\int_{x_{*}}^{\infty}[s'(x)-\dfrac{\dd(\widetilde{w_{m}^{0}}(x)\sqrt{1+s'(x)^{2}})}{\dd x}]\phi(x,s(x))\,\dd x\\
&\qquad+\int_{x_{*}}^{\infty}(\widetilde{w_{n}^{0}}(x)-s'(x)\widetilde{w_{m}^{0}}(x))\sqrt{1+s'(x)^{2}}\partial_{y}\phi(x,s(x))\,\dd x=0.
\end{aligned}
\end{equation}
Since $\phi$ is arbitrary,  we get not only  \eqref{2.13} and \eqref{2.14}, but also
\begin{equation}\label{3.07}
\begin{aligned}
&\widetilde{w_{m}^{0}}(x_{*})\sqrt{1+s'(x_{*})^{2}}=w_{m}^{0}(x_{*})\sqrt{1+b'(x_{*})^{2}}=b(x_{*}), \\
&\dfrac{\dd(\widetilde{w_{m}^{0}}(x)\sqrt{1+s'(x)^{2}})}{\dd x}=s'(x), \  \widetilde{w_{n}^{0}}(x)=s'(x)\widetilde{w_{m}^{0}}(x),   \ \  x>x_{*},
\end{aligned}
\end{equation}
hence
\begin{equation}\label{38}
\begin{aligned}
&\widetilde{w_{m}^{0}}(x)=\dfrac{s(x)}{\sqrt{1+s'(x)^{2}}}, \ \ \  \widetilde{w_{n}^{0}}(x)=\dfrac{s'(x)s(x)}{\sqrt{1+s'(x)^{2}}},   \ \  x>x_{*}.
\end{aligned}
\end{equation}
Similar calculation also yields
\begin{equation}
\begin{aligned}
&\widetilde{w_{m}^{3}}(x)=\dfrac{E_{0}s(x)}{\sqrt{1+s'(x)^{2}}},\quad
\widetilde{w_{n}^{3}}(x)=\dfrac{E_{0}s'(x)s(x)}{\sqrt{1+s'(x)^{2}}}.
\end{aligned}
\end{equation}

\begin{remark}\label{rm31}
It is important to notice that  the two equations in \eqref{38} imply the slip condition \eqref{1.9} on the free layer, namely \eqref{1.9} is  natural  on the free boundary $S$.
\end{remark}

By virtue of  {$$w_{p}n_{1}\delta_{\widehat{W}}=w^{f}_{p}(x)n_{1}(x)\delta_{W_{f}} +w^{e}_{p}(y)n_{1}\delta_{W_{e}} =w^{f}_{p}(x)\dfrac{-b'(x)}{\sqrt{1+b'(x_{*})^{2}}}\delta_{W_{f}} +w^{e}_{p}(y)\delta_{W_{e}},$$}
from \eqref{3.7} and \eqref{2.4}, removing the test function $\phi$, we obtain that
\begin{equation}\label{3.12}
\begin{aligned}
&w_{m}^{1}(0)=0, \ w_{n}^{1}(0)=0, \ w_{n}^{1}(x)=b'(x) w_{m}^{1}(x), \ \dfrac{\dd(w_{m}^{1}(x)\sqrt{1+b'(x)^{2}})}{\dd x}=(1-w_{p}(x))b'(x),\\
&\widetilde{w_{m}^{1}}(x_{*})\sqrt{1+s'(x_{*})^{2}}=w_{m}^{1}(x_{*})\sqrt{1+b'(x_{*})^{2}} =\dfrac{H(x_{*})}{\sqrt{1+b'(x_{*})^{2}}},
 \\ &\widetilde{w_{n}^{1}}(x)=s'(x)\widetilde{w_{m}^{1}}(x), \ \ w^{e}_{p}(y)=\underline{p}, \ \  \dfrac{\dd(\widetilde{w_{m}^{1}}(x)\sqrt{1+s'(x)^{2}})}{\dd x}=(1-\underline{p})s'(x),
\end{aligned}
\end{equation}
from which we solve
\begin{equation}\label{311new}
\begin{aligned}
\widetilde{w_{m}^{1}}(x)=\dfrac{(1-\underline{p})[s(x)-b(x_{*})]+\dfrac{H(x_{*})}{\sqrt{1+b'(x_{*})^{2}}}}{\sqrt{1+s'(x)^{2}}},\qquad x\ge x_*.
\end{aligned}
\end{equation}
Similarly noticing that $$w_{p}n_{2}\delta_{\widehat{W}}=w^{f}_{p}(x)n_{2}\delta_{W_{f}}+w^{e}_{p}(y)n_{2}\delta_{W_{e}}
=\dfrac{w^{f}_{p}(x)}{\sqrt{1+b'(x)^{2}}}\delta_{W_{f}},$$
thanks to \eqref{3.10} and \eqref{2.5}, we find
\begin{align}
&w_{m}^{2}(0)=0,\ \ w_{n}^{2}(x)=b'(x) w_{m}^{2}(x), \ \ \dfrac{\dd(w_{m}^{2}(x)\sqrt{1+b'(x)^{2}})}{\dd x}=w^{f}_{p}(x),\label{311}\\
&\widetilde{w_{m}^{2}}(x_{*})\sqrt{1+s'(x_{*})^{2}}=w_{m}^{2}(x_{*})
\sqrt{1+b'(x_{*})^{2}}=\dfrac{b'(x_{*})H(x_{*})}{\sqrt{1+b'(x_{*})^{2}}},
\label{312}\\  &\widetilde{w_{n}^{2}}(x)=s'(x)\widetilde{w_{m}^{2}}(x),\ \   \dfrac{\dd(\widetilde{w_{m}^{2}}(x)\sqrt{1+s'(x)^{2}})}{\dd x}=\underline{p}.\label{313}
\end{align}
Note that \eqref{311}, as we expected, is the same as \eqref{2.22} and \eqref{2.23}. By
\eqref{312}\eqref{313}, we discover
\begin{equation}\label{3.12}
\begin{aligned}
&\widetilde{w_{m}^{2}}(x)=\dfrac{\underline{p}(x-x_{*})+\dfrac{b'(x_{*})H(x_{*})}{\sqrt{1+b'(x_{*})^{2}}}}{\sqrt{1+s'(x)^{2}}},\qquad x\ge x_*.
\end{aligned}
\end{equation}
According to \eqref{2.8}, from \eqref{38}, \eqref{311new} and \eqref{3.12}, we obtain
\begin{equation}
\begin{aligned}
&u|_{S}=\dfrac{(1-\underline{p})[s(x)-b(x_{*})]+\dfrac{H(x_{*})}{\sqrt{1+b'(x_{*})^{2}}}}{s(x)},\\ &v|_{S}=\dfrac{\underline{p}(x-x_{*})+\dfrac{b'(x_{*})H(x_{*})}{\sqrt{1+b'(x_{*})^{2}}}}{s(x)}.
\end{aligned}
\end{equation}
Combining with \eqref{2.30}, we have
\begin{equation}\label{3.14}
\begin{aligned}
&u=\mathsf{I}_{\Omega_{1}}+\dfrac{H(x)}{b(x)\sqrt{1+b'(x)^{2}}}\mathsf{I}_{W_{f}}+\dfrac{(1-\underline{p})[s(x)-b(x_{*})]+\dfrac{H(x_{*})}{\sqrt{1+b'(x_{*})^{2}}}}{s(x)}\mathsf{I}_{S},
\\ &v=\dfrac{b'(x)H(x)}{b(x)\sqrt{1+b'(x)^{2}}}\mathsf{I}_{W_{f}}+\dfrac{\underline{p}(x-x_{*})+\dfrac{b'(x_{*})H(x_{*})}{\sqrt{1+b'(x_{*})^{2}}}}{s(x)}\mathsf{I}_{S},\\
&{E=E_{0}\mathsf{I}_{\Omega_{1}}+\underline{E}\mathsf{I}_{\Omega_{2}}+E_{0}\mathsf{I}_{W_{f}}+E_{0}\mathsf{I}_{S}}.
\end{aligned}
\end{equation}
Hence the measure of density is
\begin{equation}\label{3.15}
\begin{aligned}
&\varrho=\mathsf{I}_{\Omega_{1}}\LL^{2}+\underline{\rho}\mathsf{I}_{\Omega_{2}}\LL^{2}+\dfrac{(b(x))^{2}}{H(x)}\mathsf{\delta}_{W_{f}}+\dfrac{(s(x))^{2}}{\sqrt{1+s'(x)^{2}}[(1-\underline{p})(s(x)-b(x_{*}))+\dfrac{H(x_{*})}{\sqrt{1+b'(x_{*})^{2}}}]}\delta_{S}.\\
\end{aligned}
\end{equation}

From \eqref{3.14} and \eqref{3.15}, if we could solve the free layer $y=s(x)$, then a measure solution to Problem 2 is determined. By the slip condition \eqref{1.9}, there is an ordinary differential equation for $s(x)$:
\begin{equation}\label{3.16}
\left \{
\begin{aligned}
&\underline{p}(x-x_{*})+\dfrac{b'(x_{*})H(x_{*})}{\sqrt{1+b'(x_{*})^{2}}}=s'(x)[(1-\underline{p})(s(x)-b(x_{*}))+\dfrac{H(x_{*})}{\sqrt{1+b'(x_{*})^{2}}}],\\
&s(x_{*})=b(x_{*}).
\end{aligned} \right.
\end{equation}
Integrating both sides of \eqref{3.16} yields
\begin{equation}\label{3.17}
\begin{aligned}
&(1-\underline{p})s^{2}(x)+2[\dfrac{H(x_{*})}{\sqrt{1+b'(x_{*})^{2}}} -(1-\underline{p})b(x_{*})]s(x)-\underline{p}(x-x_{*})^{2}\\
&-\dfrac{2b'(x_{*})H(x_{*})}{\sqrt{1+b'(x_{*})^{2}}}(x-x_{*})+b(x_{*})[(1-\underline{p})b(x_{*}) -\dfrac{2H(x_{*})}{\sqrt{1+b'(x_{*})^{2}}}]=0,
\end{aligned}
\end{equation}
where $s(x_{*})=b(x_{*}).$ Obviously, solution of $s(x)$ depends on the value of $\underline{p}$. We will discuss this in the next subsection.

\subsection{The shape of free layer depending on  pressure  of downward static gas}\label{sec32}

We divide the analysis into four cases, namely $\underline{p}=1$, $\underline{p}\in(0,1)$,  $\underline{p}=0$, and
$\underline{p}>1$.

\subsubsection{Case 1: $\underline{p}=1$}\label{sec321}
For  this simple  case,  from \eqref{3.17} we easily  see
\begin{equation}\label{3.32}
\begin{aligned}
&s(x)=\dfrac{\sqrt{1+b'(x_{*})^{2}}(x-x_{*})^{2}}{2H(x_{*})}+b'(x_{*})(x-x_{*}) +b(x_{*}),\quad \forall x\ge x_*.
\end{aligned}
\end{equation}

\medskip
In the following we focus on the more involved situation that
$\underline{p}\neq 1$.  We solve from \eqref{3.17} that
\begin{equation}\label{3.18}
s(x)=\dfrac{\sqrt{\Delta}}{1-\underline{p}}
+b(x_{*})-\dfrac{H(x_{*})}{\sqrt{1+b'(x_{*})^{2}}(1-\underline{p})}.
\end{equation}
To make sure \eqref{3.18} is meaningful, it shall holds
\begin{equation}\label{3.19}
\Delta\triangleq (1-\underline{p})\underline{p}(x-x_{*})^{2}+\dfrac{2(1-\underline{p})H(x_{*})b'(x_{*})}{\sqrt{1+b'(x_{*})^{2}}}(x-x_{*})+(\dfrac{H(x_{*})}{\sqrt{1+b'(x_{*})^{2}}})^{2}\geq 0, \ \ x\geq x_{*}.
\end{equation}

\subsubsection{Case 2: $\underline{p}\in(0,1)$}\label{sec322}

For this case, $(1-\underline{p})\underline{p}>0$, hence the terms in \eqref{3.19} are always nonnegative for $x\ge x_*$, thanks to the assumption that $b'(x)\ge0$ and $H(x_*)>0$. Therefore for this case the solution of \eqref{3.16} is given by
\eqref{3.18}.

Furthermore, from \eqref{3.18} we see that $s(x)$ is of the order $\sqrt{\dfrac{\underline{p}}{1-\underline{p}}}x$ as $x\to\infty$.

\begin{remark}
As a special case, for $\underline{p}=\dfrac{(b'(x_ {*}))^{2}}{1+(b'(x_{*}))^{2}}$, the free layer is the straight line
$s(x)=b'(x_{*})(x-x_{*})+b(x_{*}).$
\end{remark}

\subsubsection{Case 3: $\underline{p}=0$}\label{sec323}

By \eqref{3.18} one has
\begin{equation}\label{3.31}
\begin{aligned}
&s(x)=\sqrt{\dfrac{2H(x_{*})b'(x_{*})}{\sqrt{1+b'(x_{*})^{2}}}(x-x_{*})+(\dfrac{H(x_{*})}{\sqrt{1+b'(x_{*})^{2}}})^{2}}+b(x_{*})-\dfrac{H(x_{*})}{\sqrt{1+b'(x_{*})^{2}}}.
\end{aligned}
\end{equation}
Therefore the free layer is of the order $\sqrt{x}$ as $x\to\infty$. In particular, if $b'(x_{*})=0$, then \eqref{3.31} implies that $s(x)=b(x_{*})$, namely the free layer is a straight line parallel to the upcoming limiting hypersonic flow and no particle impinges on the free layer.

\begin{remark}\label{rm32}
We note that this case includes the situation that the state below the free layer is vacuum.
\end{remark}

\begin{remark}\label{rm33}
We have shown in Remark \ref{rm22} nonexistence of nontrivial ramp that suffices no force from the limiting hypersonic flow. However, the above results show that, we could design a forebody with boundary $W_f$ to form a shock layer, and then design an afterbody with boundary \eqref{3.31} so that there is no force acting on the afterbody; for this case the free layer becomes a shield which bears all the force from the limiting hypersonic flow. We note that such ideas have already appeared in discussions in \cite[p.142]{Hayes2004hypersonic}.
\end{remark}

\subsubsection{Case 4: $\underline{p}>1$}\label{sec324}

For this case, $(1-\underline{p})\underline{p}<0$ and some singularity will appear if we still assume that the free layer is a graph of a function of $x$. So we use a parametric representation $(x(t),y(t))$ of the free layer, with $t$ the parameter satisfying $x(t_{*})={ x_{*}, y(t_{*})=b(x_{*})}$. Substituting this into the definition of measure solutions,  analysis like before shows that
\begin{eqnarray}
&&\begin{cases}
\frac{\widetilde{w_{n}^{0}}(t_{*})\sqrt{x'(t_{*})^{2}+y'(t_{*})^{2}}}{y'(t_{*})}=w_{m}^{0}(x_{*})\sqrt{1+b'(x_{*})^{2}}=b(x_{*}),&\\
\widetilde{w_{n}^{0}}(t)x'(t)=y'(t)\widetilde{w_{m}^{0}}(t), \ \  \frac{\widetilde{w_{n}^{0}}(t)\sqrt{x'(t)^{2}+y'(t)^{2}}}{y'(t)}=y(t),
&  t>t_{*},
\end{cases}\\
&&\begin{cases}
\frac{\widetilde{w_{n}^{1}}(t_{*})\sqrt{x'(t_{*})^{2}+y'(t_{*})^{2}}}{y'(t_{*})}=\frac{H(x_{*})}{\sqrt{1+b'(x_{*})^{2}}}, \
 \widetilde{w_{n}^{1}}(t)x'(t)=y'(t)\widetilde{w_{m}^{1}}(t), &\\
\frac{\widetilde{w_{n}^{1}}(t)\sqrt{x'(t)^{2}+y'(t)^{2}}}{y'(t)}=(1-\underline{p})(y(t)-b(x_{*}))+\frac{H(x_{*})}{\sqrt{1+b'(x_{*})^{2}}},  & t>t_{*},
\end{cases}\\
&&\begin{cases}
\frac{\widetilde{w_{n}^{2}}(t_{*})\sqrt{x'(t_{*})^{2}+y'(t_{*})^{2}}}{y'(t_{*})}=\frac{b'(x_{*})H(x_{*})}{\sqrt{1+b'(x_{*})^{2}}},  \
 \widetilde{w_{n}^{2}}(t)x'(t)=y'(t)\widetilde{w_{m}^{2}}(t), &\\
\frac{\widetilde{w_{n}^{2}}(t)\sqrt{x'(t)^{2}+y'(t)^{2}}}{y'(t)}=\underline{p}(x(t)-x_{*})+\frac{b'(x_{*})H(x_{*})}{\sqrt{1+b'(x_{*})^{2}}}, & t>t_{*}.
\end{cases}
\end{eqnarray}
Hence we find for $t\ge t_*$,
\begin{equation}\label{3.23}
\begin{aligned}
&u|_{S}=\dfrac{(1-\underline{p})[y(t)-b(x_{*})]+\dfrac{H(x_{*})}{\sqrt{1+b'(x_{*})^{2}}}}{y(t)},
\\ &v|_{S}=\dfrac{\underline{p}(x(t)-x_{*})+\dfrac{b'(x_{*})H(x_{*})}{\sqrt{1+b'(x_{*})^{2}}}}{y(t)},
\end{aligned}
\end{equation}
and
\begin{equation}\label{3.24}
\begin{aligned}
&\varrho=\mathsf{I}_{\Omega_{1}}\LL^{2}+\underline{\rho}\mathsf{I}_{\Omega_{2}}\LL^{2}
+w_\rho^f(x)\mathsf{\delta}_{W_{f}}+w_\rho^S(t)\delta_{S},
\end{aligned}
\end{equation}
where $$w_\rho^f(x)=\dfrac{(b(x))^{2}}{H(x)},\qquad 0\le x\le x_*$$ and $$w_\rho^S(t)=\dfrac{y'(t)(y(t))^{2}}
{\sqrt{x'(t)^{2}+y'(t)^{2}}[\underline{p}(x(t)-x_{*})+\displaystyle\frac{b'(x_{*})H(x_{*})}{\sqrt{1+b'(x_{*})^{2}}}]}, \qquad t\ge t_*,$$
while $(x(t),y(t))$ ($t\ge t_*$) solve the following ordinary differential equations
\begin{equation}\label{3.25}
\left \{
\begin{aligned}
&[\underline{p}(x(t)-x_{*})+\dfrac{b'(x_{*})H(x_{*})}{\sqrt{1+b'(x_{*})^{2}}}]x'(t)=
y'(t)[(1-\underline{p})(y(t)-b(x_{*}))+\dfrac{H(x_{*})}{\sqrt{1+b'(x_{*})^{2}}}],\\
&y(t_{*})=b(x_{*}).
\end{aligned} \right.
\end{equation}
Direct integration shows the solution is an ellipse passing $(x_*, b(x_*))$:
\begin{multline}\label{331}
\underline{p}\left[x(t)-x_*+\frac{1}{\underline{p}}\frac{b'(x_*)H(x_*)}{\sqrt{1+b'(x_*)^2}}\right]^2+
(\underline{p}-1)\left[y(t)-b(x_*)-\frac{1}{\underline{p}-1}\frac{H(x_*)}{\sqrt{1+b'(x_*)^2}}\right]^2
\\=\left[\frac{1}{\underline{p}}b'(x_*)^2+\frac{1}{\underline{p}-1}\right]\frac{H(x_*)^2}{1+b'(x_*)^2}.
\end{multline}
At the right-most point $(x_{\vartriangle}, y_{\vartriangle})$, where
\begin{align}\label{3.28}
&x_{\vartriangle}=x(t_{\vartriangle})=x_{*}+\dfrac{H(x_{*})}{\underline{p}\sqrt{1+b'(x_{*})^{2}}}\left(\sqrt{b'(x_{*})^{2}
+\dfrac{\underline{p}}{\underline{p}-1}}-b'(x_{*})\right),\\
&y_{\vartriangle}=y(t_{\vartriangle})=b(x_{*})+\dfrac{H(x_{*})}{\sqrt{1 +b'(x_{*})^{2}}(\underline{p}-1)},\label{3.300}
\end{align}
we have
\begin{equation}\label{3.30}
\begin{aligned}
&u\lvert_{(x_{\vartriangle},\ y_{\vartriangle})}=0, \ \ \  v\lvert_{(x_{\vartriangle},\ y_{\vartriangle})}
={\frac{\dfrac{H(x_{*})}{\sqrt{1+b'(x_{*})^{2}}}\sqrt{b'(x_{*})^{2}
+\dfrac{\underline{p}}{\underline{p}-1}}}{b(x_{*})+\dfrac{H(x_{*})}{\sqrt{1 +b'(x_{*})^{2}}(\underline{p}-1)}}}>0,
\end{aligned}
\end{equation}
and the weight
$$w_\rho^S(t_{\vartriangle})={\displaystyle \frac{y_{\vartriangle}^{2}}{\dfrac{H(x_{*})}{\sqrt{1+b'(x_{*})^{2}}}\sqrt{b'(x_{*})^{2}
+\dfrac{\underline{p}}{\underline{p}-1}}}}<\infty.$$

Since we are treating a hyperbolic problem upside of the free layer, with the positive $x$-axis the hyperbolic direction, by causality, the upper branch of the ellipse \eqref{331} shall no longer be a part of the free layer. Also, noticing that $u(x_{\vartriangle}, y)=0$ holds only at $y=y_{\vartriangle}$, the line segment $x=x_{\vartriangle}$ with $y\ge y_{\vartriangle}$ cannot be a free layer. So we conclude that the measure solution blows up at the point $(x_{\vartriangle}, y_{\vartriangle})$, in the sense that the free layer satisfying \eqref{3.30} and cannot be prolonged anymore.

\subsection{Conclusion and examples}\label{sec33}
In summary, we have the following lemma.
\begin{lemma}

  i) \ For $\underline{p}=1$, the free layer takes the form \eqref{3.32}.

  ii) \ For $0\le\underline{p}<1$, the free layer $y=s(x)$ is given by \eqref{3.18}, defined for all $x\ge x_{*}$.

    iii) \ For $\underline{p}>1$, the free layer exists only for $x_{*}\leq x\leq x_\vartriangle$ and ending at the point $(x_{\vartriangle}, y_{\vartriangle})$, where it rolls up and can not be prolonged. So in such a sense the measure solution blows up.
\end{lemma}
This finishes  proof of Theorem \ref{thm1.2}.

\begin{figure}[H]
\centering

\subfigure[$\underline{p}=0,
	s(x)=\sqrt{\dfrac{2}{3}x-\dfrac{4}{9}}+\dfrac{\sqrt{2}}{3}.$]{
\begin{minipage}[t]{0.5\linewidth}
\centering
\includegraphics[width=2in]{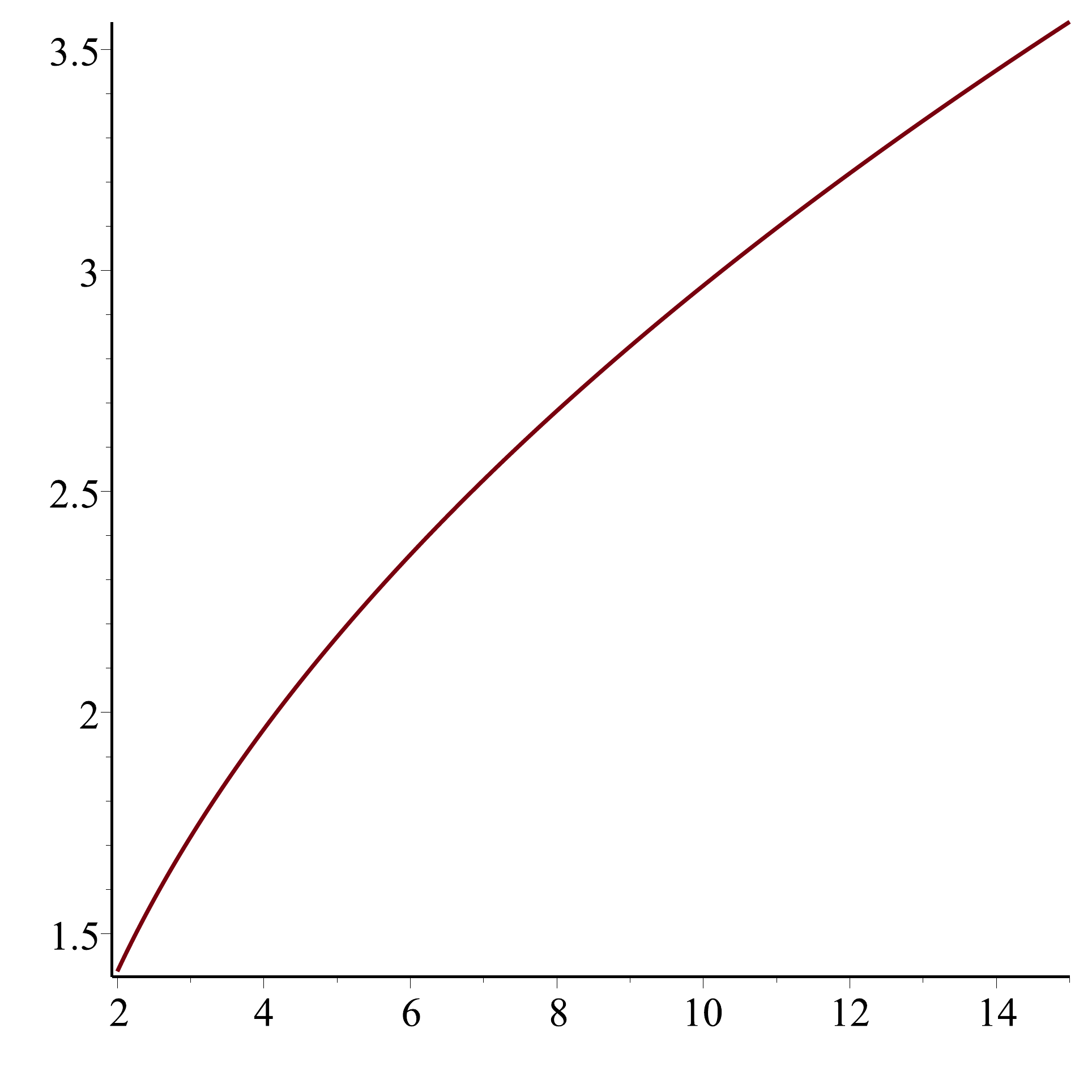}
\end{minipage}%
}%
\subfigure[$\underline{p}=\dfrac{1}{2},
	s(x)=2\sqrt{\dfrac{1}{4}x^2-\dfrac{2}{3}x+\dfrac{11}{9}}-\dfrac{\sqrt{2}}{3}.$]{
\begin{minipage}[t]{0.5\linewidth}
\centering
\includegraphics[width=2in]{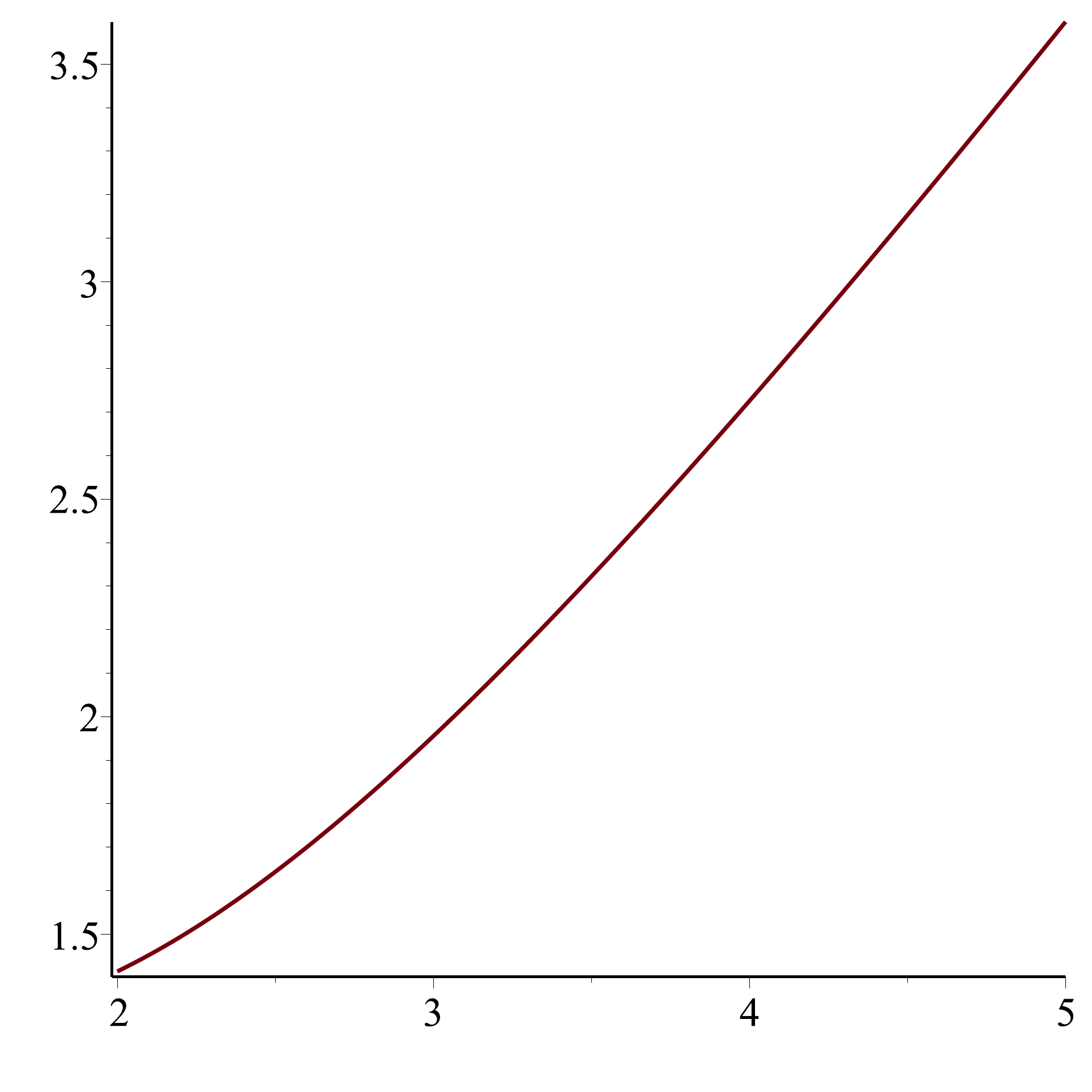}
\end{minipage}%
}%

\subfigure[$\underline{p}=1,
	s(x)=\dfrac{3}{4\sqrt{2}}x^2-\dfrac{5}{2\sqrt{2}}x+2\sqrt{2}.$]{
\begin{minipage}[t]{0.5\linewidth}
\centering
\includegraphics[width=2in]{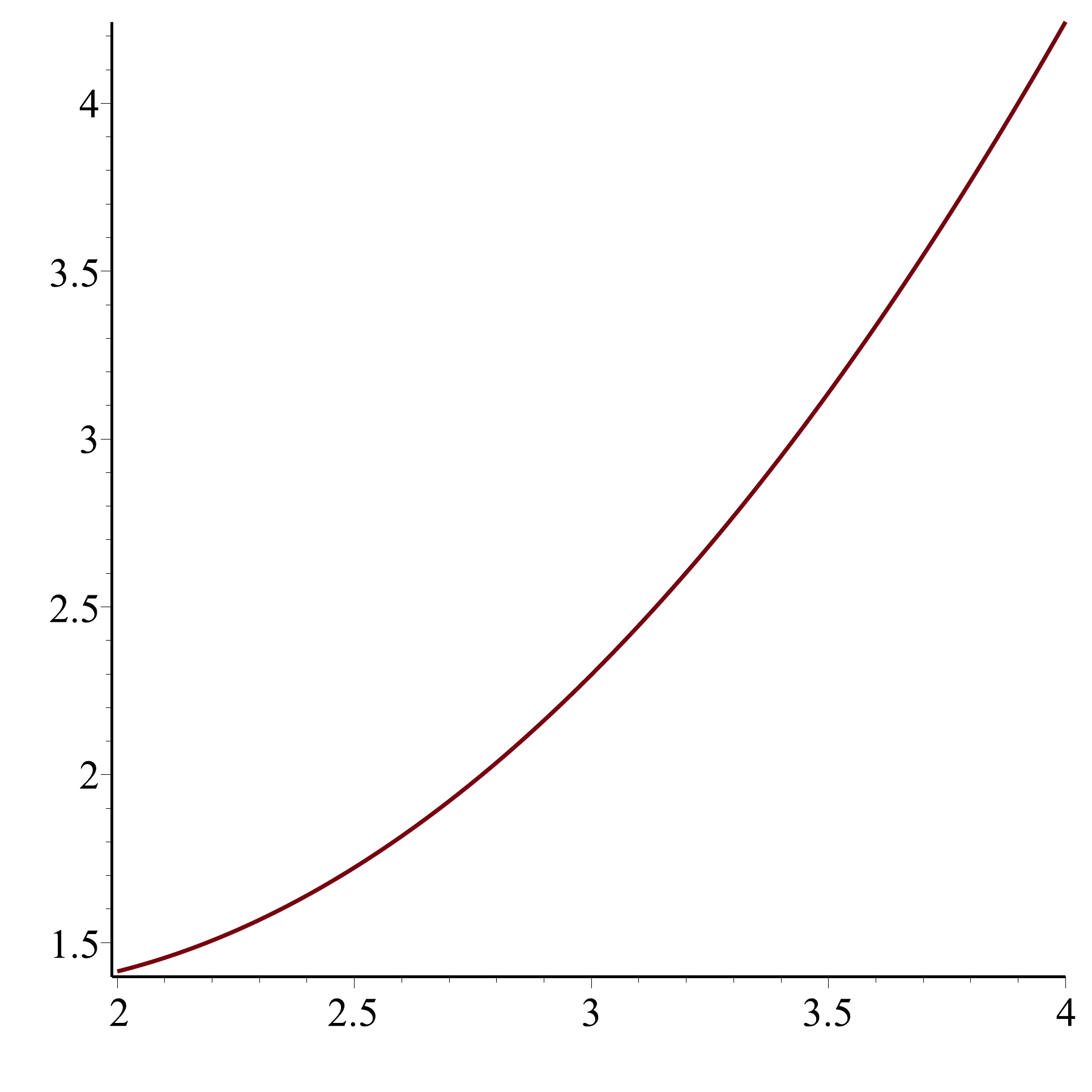}
\end{minipage}
}%
\subfigure[$\underline{p}=2,
		s(x)=-\sqrt{-2x^2+\dfrac{22x}{3}-52/9}+\dfrac{5\sqrt{2}}{3}.$]{
\begin{minipage}[t]{0.5\linewidth}
\centering
\includegraphics[width=2in]{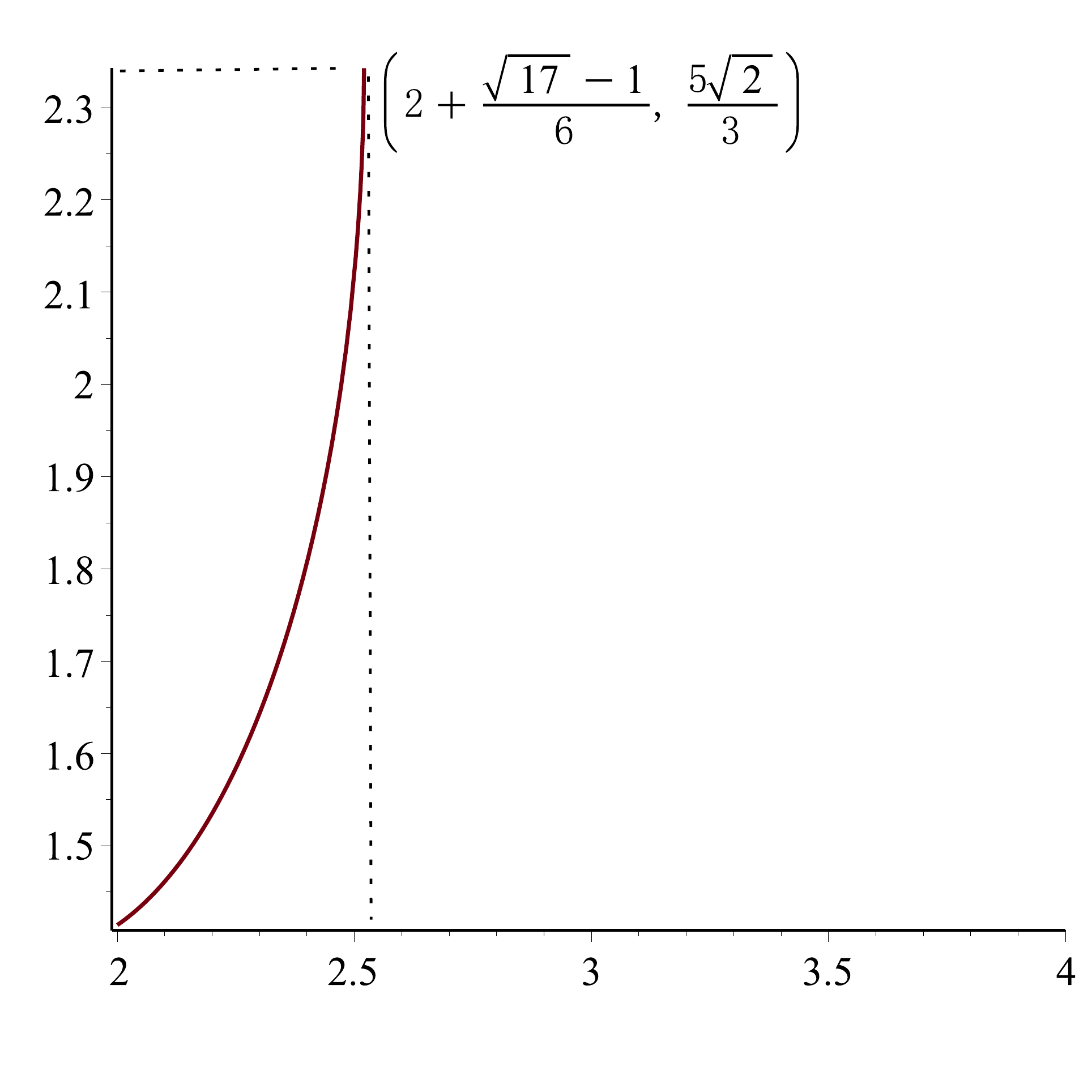}
\end{minipage}
}%

\centering
\caption{(a)(b)(c)(d) demonstrate the free layer for different pressure $\underline{p}$, with $b(x)=\sqrt{x}$ and $x_{*}=2$ being fixed.  So $(2,\sqrt{2})$ is the starting point of the free layer. In (d) the point
$(2+\dfrac{\sqrt{17}-1}{6},\dfrac{5\sqrt{2}}{3})$ is where the free layer rolls up.}\label{fig5}
\end{figure}

\begin{remark}
As an example, we take $b(x)=\sqrt{x}$, $x_{*}=2$ to draw graphs of free layers, with different pressure of the downward static gas. For $\underline{p}=0$, we have  $s(x)=\sqrt{\dfrac{2}{3}x-\dfrac{4}{9}}+\dfrac{\sqrt{2}}{3}$ (see Figure \ref{fig5}(a)).
For $\underline{p}=\dfrac{1}{2}$,
$s(x)=2\sqrt{\dfrac{1}{4}(x-2)^2+\dfrac{1}{3}x+\dfrac{2}{9}}-\dfrac{\sqrt{2}}{3}$
(see Figure \ref{fig5}(b)).
For $\underline{p}=1$, $s(x)=\dfrac{3} {4\sqrt{2}}(x-2)^2+\dfrac{1}{2\sqrt{2}}(x-2)+\sqrt{2}$ (see Figure \ref{fig5}(c)). For $\underline{p}=2$,
$s(x)=-\sqrt{-2(x-2)^2-\dfrac{2x}{3}+20/9}+\dfrac{5\sqrt{2}}{3}$ and $(x_{\vartriangle}, s(x_{\vartriangle}))=(2+\dfrac{\sqrt{17}-1}{6},\dfrac{5\sqrt{2}}{3})$, where the free layer rolls up  (see Figure \ref{fig5}(d)).

We also draw these graphs in the same frame for comparison, see Figure \ref{fig8}.
\vspace{-3cm}
\begin{figure}[H]
\centering
\includegraphics[width=6in]{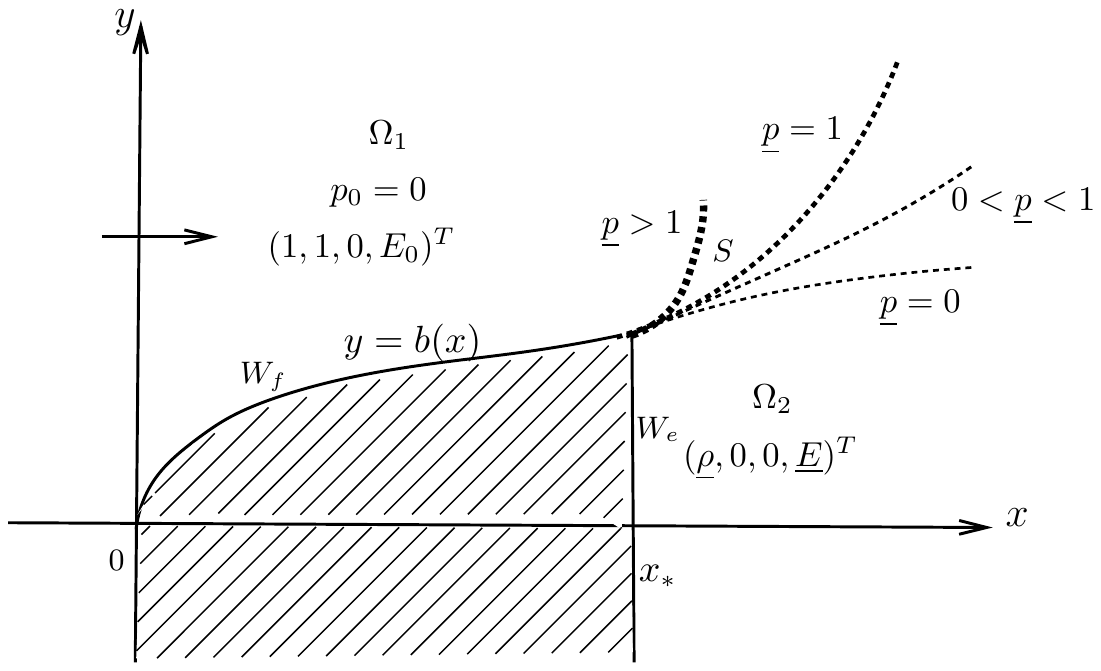}\\
\vspace{-13cm}\caption {Different shapes of free layers for different pressures in ``dead gas zone".}\label{fig8}
\end{figure}
\end{remark}

\section {Interactions of limiting hypersonic flows and pressureless jets}\label{sec4}

In this section we study Problem 3. We may  define its measure solution in the sprit of Definition \ref{def2.2}. Since there are initial data on $x=x_*$, item i) in Definition \ref{def2.2} shall be replaced by
\begin{eqnarray}
&&\langle m^{0},\partial_{x}\phi\rangle+\langle n^{0},\partial_{y}\phi\rangle+\int_{0}^{\infty}\rho_{0}u_{0}\phi(0,y)\,\dd y +\int_{-\infty}^{b(x_{*})}\underline{\rho}\underline{u} \phi(x_{*},y)\,\dd y=0,\label{4.1}\\
&&\langle m^{1},\partial_{x}\phi\rangle+\langle n^{1},\partial_{y}\phi\rangle+\langle\wp,\partial_{x}\phi\rangle+\langle w_{p}n_{1}\delta_{W},\phi\rangle+\int_{0}^{\infty}(\rho_{0}u_{0}^{2}+p_{0})\phi(0,y)\,\dd y\nonumber\\		&&\qquad\qquad\qquad+\int_{-\infty}^{b(x_{*})}(\underline{\rho}\underline{u}^{2}+\underline{p}) \phi(x_{*},y)\,\dd y=0,\label{4.2}\\
&&\langle m^{2},\partial_{x}\phi\rangle +\langle n^{2},\partial_{y}\phi\rangle+\langle\wp,\partial_{y}\phi\rangle+\langle w_{p}n_{2}\delta_{W},\phi\rangle+\int_{0}^{\infty}\rho_{0}u_{0}v_{0}\phi(0,y)\,\dd y\nonumber\\ &&\qquad\qquad\qquad+\int_{-\infty}^{b(x_{*})}\underline{\rho}\underline{u} \ \underline{v} \phi(x_{*},y)\,\dd y=0,\label{4.3}\\
&&\langle m^{3},\partial_{x}\phi\rangle+\langle n^{3},\partial_{y}\phi\rangle+\int_{0}^{\infty}\rho_{0}u_{0}E_{0}\phi(0,y)\,\dd y +\int_{-\infty}^{b(x_{*})}(\underline{\rho}\underline{u} \ \underline{E}) \phi(x_{*},y)\,\dd y=0,\label{4.4}
\end{eqnarray}
where $\phi\in C_{0}^{1}(\Bbb R^{2})$ is an arbitrary test function. The other requirements in Definition \ref{def2.2} are unchanged. 	

To construct a measure solution with physical significance, we need the following {\emph{entropy condition}}:
\begin{equation}\label{45}
\frac{\underline{v}}{\underline{u}}\ge \frac{v|_S}{u|_S}=s'(x)\ge \frac{v_0}{u_0},
 \end{equation}
where $S: y=s(x)$ is a free layer separating constant state $U_0$ lying above it and constant state $\underline{U}$ below it. It is a generalization of Lax entropy conditions for shocks, and is widely used in the studies of delta shocks, see, for example, \cite[p.329, (3.5)]{Cheng2011Delta}, or below (2.5) in \cite[p.749]{shensun2009} (for the unsteady case). It excludes some anomalous measure solutions such as the case that particles escape from the free layer.
To fulfill \eqref{45}, the analysis below is separated into two subsections.

\subsection{Measure solution consists of two piecewise-constant states without vacuum}\label{sec41}

For this case we assume that
\begin{eqnarray}
&&\begin{aligned}\label{4.5}
&m^{0}=\mathsf{I}_{\Omega_{1}}\mathcal{L}^{2}+
\underline{\rho}\underline{u}\mathsf{I}_{\Omega_{2}}\mathcal{L}^{2}+w_{m}^{0}(x)\delta_{\widetilde{W}}
=\mathsf{I}_{\Omega_{1}}\mathcal{L}^{2}+\underline{\rho}\underline{u}\mathsf{I}_{\Omega_{2}}\mathcal{L}^{2}+w_{m}^{0}(x)\delta_{W_{f}}
+\widetilde{w_{m}^{0}}(x)\delta_{S}, \\
&n^{0}=\underline{\rho}\underline{v}\mathsf{I}_{\Omega_{2}}\mathcal{L}^{2}+w_{n}^{0}(x)\delta_{\widetilde{W}}
=\underline{\rho}\underline{v}\mathsf{I}_{\Omega_{2}}\mathcal{L}^{2}+w_{n}^{0}(x)\delta_{W_{f}}+\widetilde{w_{n}^{0}}(x)\delta_{S},
\end{aligned}\\
 &&\begin{aligned}\label{4.9}
&m^{1}=\mathsf{I}_{\Omega_{1}}\mathcal{L}^{2}+\underline{\rho}\underline{u}^{2}\mathsf{I}_{\Omega_{2}}\mathcal{L}^{2}
+w_{m}^{1}(x)\delta_{\widetilde{W}}=\mathsf{I}_{\Omega_{1}}\mathcal{L}^{2}+\underline{\rho}\underline{u}^{2}\mathsf{I}_{\Omega_{2}}\mathcal{L}^{2}
+w_{m}^{1}(x)\delta_{W_{f}}+\widetilde{w_{m}^{1}}(x)\delta_{S},\\
&n^{1}=\underline{\rho}\underline{u} \ \underline{v}\mathsf{I}_{\Omega_{2}}\mathcal{L}^{2}+w_{n}^{1}(x)\delta_{\widetilde{W}}
=\underline{\rho}\underline{u} \ \underline{v}\mathsf{I}_{\Omega_{2}}\mathcal{L}^{2}+w_{n}^{1}(x)\delta_{W_{f}}+\widetilde{w_{n}^{1}}(x)\delta_{S},
\end{aligned}\\
&&\begin{aligned}\label{4.12}
&m^{2}=\underline{\rho}\underline{u}\, \underline{v}\mathsf{I}_{\Omega_{2}}\mathcal{L}^{2}+w_{m}^{2}(x)\delta_{\widetilde{W}}
=\underline{\rho}\underline{u}\, \underline{v}\mathsf{I}_{\Omega_{2}}\mathcal{L}^{2}+w_{m}^{2}(x)\delta_{W_{f}}+\widetilde{w_{m}^{2}}(x)\delta_{S},
\\ &n^{2}=\underline{\rho}\underline{v}^{2}\mathsf{I}_{\Omega_{2}}\mathcal{L}^{2}+w_{n}^{2}(x)\delta_{\widetilde{W}}
=\underline{\rho}\underline{v}^{2}\mathsf{I}_{\Omega_{2}}\mathcal{L}^{2}+w_{n}^{2}(x)\delta_{W_{f}}+\widetilde{w_{n}^{2}}(x)\delta_{S},  \ \ \  \wp=0,
\end{aligned}\\
&&\begin{aligned}
&m^{3}=E_{0}\mathsf{I}_{\Omega_{1}}\mathcal{L}^{2}+\underline{\rho}\underline{u}\, \underline{E}\mathsf{I}_{\Omega_{2}}\mathcal{L}^{2}
+w_{m}^{3}(x)\delta_{W_{f}}+\widetilde{w_{m}^{3}}(x)\delta_{S}, \\
&n^{3}=\underline{\rho}\underline{v} \ \underline{E}\mathsf{I}_{\Omega_{2}}\mathcal{L}^{2}+w_{n}^{3}(x)\delta_{W_{f}}+\widetilde{w_{n}^{3}}(x)\delta_{S},
\end{aligned}
\end{eqnarray}
where $\widetilde{w_{m}^{i}}(x),\widetilde{w_{n}^{i}}(x) \ (i=0,1,2,3)$ are functions to be solved.

\medskip
The computation mimics previous sections. Substituting \eqref{4.5} into \eqref{4.1},  we get for $x\in[0, x_*]$,
\begin{equation}
w_{m}^{0}(0)=0, \ \ w_{n}^{0}(x)=w_{m}^{0}(x)b'(x), \ \  \dfrac{\dd(w_{m}^{0}(x)\sqrt{1+b'(x)^{2}})}{\dd x}=b'(x),
\end{equation}
and for $x\ge x_*$,
\begin{equation}
\begin{aligned}
&\widetilde{w_{m}^{0}}(x_{*})\sqrt{1+s'(x_{*})^{2}}=w_{m}^{0}(x_{*})\sqrt{1+b'(x_{*})^{2}}=b(x_{*}), \\
&\dfrac{\dd(\widetilde{w_{m}^{0}}(x)\sqrt{1+s'(x)^{2}})}{\dd x}=\underline{\rho}\underline{v}+(1-\underline{\rho}\underline{u})s'(x), \  \widetilde{w_{n}^{0}}(x)=s'(x)\widetilde{w_{m}^{0}}(x).
\end{aligned}
\end{equation}
Hence
\begin{equation}
\begin{aligned}
&\widetilde{w_{m}^{0}}(x)=\dfrac{\underline{\rho}\underline{v}(x-x_{*})-\underline{\rho}\underline{u}(s(x)-b(x_{*}))+s(x)}{\sqrt{1+s'(x)^{2}}}, \\
&\widetilde{w_{n}^{0}}(x)=\dfrac{\underline{\rho}\underline{v}(x-x_{*})-\underline{\rho}\underline{u}(s(x)-b(x_{*}))+s(x)}{\sqrt{1+s'(x)^{2}}}s'(x),
\end{aligned}
\end{equation}
and similarly
\begin{equation}
\begin{aligned}
\widetilde{w_{m}^{3}}(x)=&\dfrac{\underline{\rho}\underline{v}\ \underline{E}(x-x_{*})-\underline{\rho}\underline{u}\, \underline{E}(s(x)-b(x_{*}))+E_{0}s(x)}{\sqrt{1+s'(x)^{2}}},\\
\widetilde{w_{n}^{3}}(x)=&\dfrac{\underline{\rho}\underline{v}\ \underline{E}(x-x_{*})-\underline{\rho}\underline{u}\, \underline{E}(s(x)-b(x_{*}))+E_{0}s(x)}{\sqrt{1+s'(x)^{2}}}s'(x).
\end{aligned}
\end{equation}
Note that these equations imply the slip condition \eqref{1.9} on free layer.

By \eqref{4.2} and \eqref{4.9}  we have
\begin{equation}
\begin{aligned}
&\widetilde{w_{m}^{1}}(x_{*})\sqrt{1+s'(x_{*})^{2}}=w_{m}^{1}(x_{*})\sqrt{1+b'(x_{*})^{2}}=\dfrac{H(x_{*})}{\sqrt{1+b'(x_{*})^{2}}},  \\ &\widetilde{w_{n}^{1}}(x)=s'(x)\widetilde{w_{m}^{1}}(x), \ \  \  \dfrac{\dd(\widetilde{w_{m}^{1}}(x)\sqrt{1+s'(x)^{2}})}{\dd x}=\underline{\rho}\underline{u}\, \underline{v}+(1-\underline{\rho}\underline{u}^{2})s'(x).
\end{aligned}
\end{equation}
Consequently for $x\ge x_*,$
\begin{equation}
\begin{aligned}
\widetilde{w_{m}^{1}}(x)=\dfrac{\underline{\rho}\underline{u}\, \underline{v}(x-x_{*})+(1-\underline{\rho}\underline{u}^{2})(s(x)-b(x_{*}))+\dfrac{H(x_{*})}{\sqrt{1+b'(x_{*})^{2}}}}{\sqrt{1+s'(x)^{2}}}.
\end{aligned}
\end{equation}
Moreover, applying \eqref{4.12}, \eqref{4.3} is reduced to
\begin{equation}
\begin{aligned}
&\widetilde{w_{m}^{2}}(x_{*})\sqrt{1+s'(x_*)^{2}}=w_{m}^{2}(x_{*})\sqrt{1+b'(x_*)^{2}}=\dfrac{b'(x_{*})H(x_{*})}{\sqrt{1+b'(x_{*})^{2}}}, \\  &\widetilde{w_{n}^{2}}(x)=s'(x)\widetilde{w_{m}^{2}}(x),\ \ \   \dfrac{\dd(\widetilde{w_{m}^{2}}(x)\sqrt{1+s'(x)^{2}})}{\dd x}=\underline{\rho}\underline{v}^{2}-\underline{\rho}\underline{u}\, \underline{v}s'(x).
\end{aligned}
\end{equation}
It follows
\begin{equation}
\begin{aligned}
&\widetilde{w_{m}^{2}}(x)=\dfrac{\underline{\rho}\underline{v}^{2}(x-x_{*})-\underline{\rho}\underline{u}\, \underline{v}(s(x)-b(x_{*}))+\dfrac{b'(x_{*})H(x_{*})}{\sqrt{1+b'(x_{*})^{2}}}}{\sqrt{1+s'(x)^{2}}},
\end{aligned}
\end{equation}
and by \eqref{2.8},
\begin{equation}
\begin{aligned}
&u|_{S}=\dfrac{\underline{\rho}\underline{u}\, \underline{v}(x-x_{*})+(1-\underline{\rho}\underline{u}^{2})(s(x)-b(x_{*}))+\dfrac{H(x_{*})}{\sqrt{1+b'(x_{*})^{2}}}}{{\underline{\rho}\underline{v}(x-x_{*})-\underline{\rho}\underline{u}(s(x)-b(x_{*}))+s(x)}},\\ &v|_{S}=\dfrac{\underline{\rho}\underline{v}^{2}(x-x_{*})-\underline{\rho}\underline{u}\, \underline{v}(s(x)-b(x_{*}))+\dfrac{b'(x_{*})H(x_{*})}{\sqrt{1+b'(x_{*})^{2}}}}{{\underline{\rho}\underline{v}(x-x_{*})-\underline{\rho}\underline{u}(s(x)-b(x_{*}))+s(x)}}.
\end{aligned}
\end{equation}
So we discover that
\begin{equation}\label{4.16}
\begin{aligned}
u=&\mathsf{I}_{\Omega_{1}}+\underline{u}\mathsf{I}_{\Omega_{2}}+\dfrac{H(x)}{b(x)\sqrt{1+b'(x)^{2}}}\mathsf{I}_{W_{f}}
\\&+\dfrac{\underline{\rho}\underline{u}\, \underline{v}(x-x_{*})+(1-\underline{\rho}\underline{u}^{2})(s(x)-b(x_{*}))
+\dfrac{H(x_{*})}{\sqrt{1+b'(x_{*})^{2}}}}{\underline{\rho}\underline{v}(x-x_{*}) -\underline{\rho}\underline{u}(s(x)-b(x_{*}))+s(x)}\mathsf{I}_{S},
\\ v=&\underline{v}\mathsf{I}_{\Omega_{2}}+\dfrac{b'(x)H(x)}{b(x)\sqrt{1+b'(x)^{2}}}\mathsf{I}_{W_{f}}
\\&+\dfrac{\underline{\rho}\underline{v}^{2}(x-x_{*})-\underline{\rho}\underline{u}\, \underline{v}(s(x)-b(x_{*}))+\dfrac{b'(x_{*})H(x_{*})}{\sqrt{1+b'(x_{*})^{2}}}}{\underline{\rho}\underline{v}(x-x_{*})-\underline{\rho}\underline{u}(s(x)-b(x_{*}))+s(x)}\mathsf{I}_{S},
\end{aligned}
\end{equation}
\begin{equation}
\begin{aligned}
E=&E_{0}\mathsf{I}_{\Omega_{1}}+\underline{E}\mathsf{I}_{\Omega_{2}}
+E_{0}\mathsf{I}_{W_{f}}\\&
+\dfrac{\underline{\rho}\underline{v}\ \underline{E}(x-x_{*})-\underline{\rho}\underline{u}\, \underline{E}(s(x)-b(x_{*}))+E_{0}s(x)}{\underline{\rho}\underline{v}(x-x_{*})-\underline{\rho}\underline{u}(s(x)-b(x_{*}))+s(x)}\mathsf{I}_{S},
\end{aligned}
\end{equation}
and the measure of density is
\begin{equation}\label{4.17}
\begin{aligned}
\varrho=&\mathsf{I}_{\Omega_{1}}\LL^{2}+\underline{\rho}\mathsf{I}_{\Omega_{2}}\LL^{2}+\dfrac{(b(x))^{2}}{H(x)}\mathsf{\delta}_{W_{f}}\\
&+\dfrac{{[\underline{\rho}\underline{v}(x-x_{*})-\underline{\rho}\underline{u}(s(x)-b(x_{*}))+s(x)]}^{2}}{\sqrt{1+s'(x)^{2}}[\underline{\rho}\underline{u}\, \underline{v}(x-x_{*})+(1-\underline{\rho}\underline{u}^{2})(s(x)-b(x_{*}))+\dfrac{H(x_{*})}{\sqrt{1+b'(x_{*})^{2}}}]}\delta_{S}.\\
\end{aligned}
\end{equation}

So what left is to determine $s(x)$.
Using the slip condition \eqref{1.9}, from \eqref{4.16}  we have the following ordinary differential equations:
\begin{equation}\label{4.18}
\left \{
\begin{aligned}
&\underline{\rho}\underline{v}^{2}(x-x_{*})-\underline{\rho}\underline{u}\, \underline{v}(s(x)-b(x_{*}))+\dfrac{b'(x_{*})H(x_{*})}{\sqrt{1+b'(x_{*})^{2}}}\\
&\qquad=s'(x)[\underline{\rho}\underline{u}\, \underline{v}(x-x_{*})+(1-\underline{\rho}\underline{u}^{2})(s(x)-b(x_{*}))+\dfrac{H(x_{*})}{\sqrt{1+b'(x_{*})^{2}}}],\\
&s(x_{*})=b(x_{*}).
\end{aligned} \right.
\end{equation}

\begin{lemma}
\noindent 1) If $1-\underline{\rho}\underline{u}^{2}=0$ and $\underline{v}/{\underline{u}}\ge b'(x_{*})$, \eqref{4.18} has a unique solution $y=s(x)$, which is of the order $\displaystyle\frac{\underline{v}}{2\underline{u}}x$ as $x\to\infty$, and {$y=s(x)$} is a free layer satisfying entropy condition \eqref{45}, separating the limiting hypersonic flow above it from the pressureless jet below it.

 \noindent 2) For $1-\underline{\rho}\underline{u}^{2}\neq0$ and $\underline{v}/{\underline{u}}\ge b'(x_{*})$,
\eqref{4.18} has a solution $y=s(x)$ which is of the order $\displaystyle \frac{\sqrt{\underline{\rho}}\underline{v}x}{1+\sqrt{\underline{\rho}}\underline{u}}$ as $x\to\infty$, and similar conclusion as in 1) also holds.
\end{lemma}

\begin{proof}

1) If $1-\underline{\rho}\underline{u}^{2}=0$,  then \eqref{4.18} is linear and we get
\begin{equation}\label{422}
\begin{aligned}
&s(x)=\dfrac{\underline{\rho}\underline{v}^{2}(x-x_{*})^{2}+2\left[\underline{\rho}\underline{u} \ \underline{v}b(x_{*})
+\dfrac{b'(x_{*})H(x_{*})}{\sqrt{1+b'(x_{*})^{2}}}\right](x-x_{*})+\dfrac{2b(x_{*})H(x_{*})}{\sqrt{1+b'(x_{*})^{2}}}}
{2\underline{\rho}\underline{u} \ \underline{v}(x-x_{*})+\dfrac{2H(x_{*})}{\sqrt{1+b'(x_{*})^{2}}}}.
\end{aligned}
\end{equation}
Since $\underline{v}/{\underline{u}}\ge b'(x_{*})\ge0$ and $H(x_*)>0$, the denominator is positive and $s(x)$ is defined for all $x\ge x_*$.
By straightforward computations, the denominator
\begin{equation}\label{423}
d(x)\triangleq\underline{\rho}\underline{v}(x-x_{*})-\underline{\rho}\underline{u}(s(x)-b(x_{*}))+s(x)
\end{equation}
 in \eqref{4.16} is also positive for all $x\ge x_*$. \footnote{For {$\underline{\rho}\underline{u}\le1$} or equivalently {$\underline{u}\ge1$}, this is obvious; for $\underline{u}\in(0,1)$, we are led to show the quadratic function $$\underline{\rho}\underline{v}^2(1+\underline{u})(x-x_*)^2+\Big[2\underline{v}b(x_*) +\frac{2H(x_*)}{\sqrt{1+b'(x_*)^2}}(\frac{\underline{v}}{\underline{u}}-b'(x_*) +\underline{u}b'(x_*))\Big](x-x_*)+{ \frac{2(1+\underline{u})b(x_*)H(x_*)}{\sqrt{1+b'(x_*)^2}}}$$
is positive for all $x\ge x_*$, which is simple.} {So does the denominator of the last term in \eqref{4.17}.} Thus the measure solution \eqref{4.16}-\eqref{4.17}\eqref{422} is well-defined.

Considering entropy condition \eqref{45}, since $v_0=0, u_0=1$, we need to check that $s'(x)\ge0$ for all $x\ge x_*$. This is equivalent to the left-hand side of \eqref{4.18} being nonnegative on $[x_*, \infty)$. The verification is also straightforward by substituting \eqref{422} into it. To show $s'(x)\le \underline{v}/\underline{u}$, we are led to prove
$$s(x)-b(x_*)\ge\underbrace{\frac{H(x_*)}{\sqrt{1+b'(x_*)^2}}\frac{\underline{u}}{\underline{v}}}_{>0}\underbrace{(b'(x_*)-\frac{\underline{v}}{\underline{u}})}_{\le0},\quad \forall x\ge x_*.$$
This is obvious since it holds at $x=x_*$, and recall we have proved that $s'(x)\ge0$ for $x\ge x_*$.

2) If $1-\underline{\rho}\underline{u}^{2}\neq 0$,  from \eqref{4.18} we have
\begin{equation}\label{4.19}
s(x)=\dfrac{(1-\underline{\rho}\underline{u}^{2})b(x_{*})-\underline{\rho}\underline{u}\,  \underline{v}(x-x_{*})-\dfrac{H(x_{*})}{\sqrt{1+b'(x_{*})^{2}}}}{1-\underline{\rho}\underline{u}^{2}} +\dfrac{\sqrt{\spadesuit}}{1 -\underline{\rho}\underline{u}^{2}}.
\end{equation}
Since $\underline{v}/{\underline{u}}\ge b'(x_{*})\ge0$, we get that $\underline{\rho}\underline{u}\,  \underline{v}+(1-\underline{\rho}\underline{u}^{2})b'(x_{*})\geq 0$. This means that for any $x\geq x_*$, we have
\begin{equation}\label{4.21}
\begin{aligned}
&\spadesuit\triangleq\underline{\rho}\underline{v}^{2}(x-x_{*})^{2}+2[\underline{\rho}\underline{u}\, \underline{v}+(1-\underline{\rho}\underline{u}^{2})b'(x_{*})]\dfrac{H(x_{*})}{\sqrt{1+b'(x_{*})^{2}}}(x -x_{*})+(\dfrac{H(x_{*})}{\sqrt{1+b'(x_{*})^{2}}})^{2}> 0.
\end{aligned}
\end{equation}
Therefore \eqref{4.19} is defined on $[x_*, \infty)$ and of the order $\displaystyle \frac{\sqrt{\underline{\rho}}\underline{v}x}{1+\sqrt{\underline{\rho}}\underline{u}}$ as $x\to\infty$.

Next it is easy to see that for all $x\ge x_*$,
\begin{eqnarray}\label{426}
\underline{\rho}\underline{u}\, \underline{v}(x-x_{*})+(1-\underline{\rho}\underline{u}^{2})(s(x)-b(x_{*})) +\dfrac{H(x_{*})}{\sqrt{1+b'(x_{*})^{2}}}=\sqrt{\spadesuit}>0.\end{eqnarray}
Supposing that $b'(x_*)>0$, to show
\begin{equation}\label{427}
\underline{\rho}\underline{v}^{2}(x-x_{*})-\underline{\rho}\underline{u}\, \underline{v}(s(x)-b(x_{*}))+\dfrac{b'(x_{*})H(x_{*})}{\sqrt{1+b'(x_{*})^{2}}}>0,
\end{equation}
straightforward computation requires
\begin{equation}
(1-\underline{\rho} \underline{u}^2)[\sqrt{\spadesuit}-\alpha(x-x_*)-\beta]<0,\quad \alpha=\frac{\underline{v}}{\underline{u}}, \ \ \beta=\frac{H(x_*)}{\sqrt{1+b'(x_*)^2}}\left(1+b'(x_*)\frac{1-\underline{\rho}\underline{u}^2} {\underline{\rho}\underline{u}\, \underline{v}}\right).
\end{equation}
Since this holds at $x=x_*$, we just need to prove that there is no root greater than $x_*$ to the equation $$\sqrt{\spadesuit}=\alpha(x-x_*)+\beta.$$
In fact, squaring the equation and after some manipulation, we have
\begin{multline*}
\left(\frac{\underline{v}}{\underline{u}}\right)^2(x-x_*)^2+2\frac{H(x_*)}{\sqrt{1+b'(x_*)^2}}
\left(\frac{\underline{v}}{\underline{u}}-b'(x_*) +\frac{1}{\underline{\rho}\underline{u}^2}b'(x_*)\right)(x-x_*)\\
+\frac{H(x_*)^2}{1+b'(x_*)^2}\frac{b'(x_*)}{\underline{\rho}\underline{u}\, \underline{v}}\left(\frac{b'(x_*)}{\underline{\rho}\underline{u}\,\underline{v}} +2-\frac{b'(x_*)\underline{u}}{\underline{v}}\right)=0,
\end{multline*}
and the claim follows by recalling that $\frac{\underline{v}}{\underline{u}}-b'(x_*)\ge0$. If $b'(x_*)=0$ and $\underline{v}>0$, we easily check that \begin{equation}\label{430add}
\underline{\rho}\underline{v}^{2}(x-x_{*})-\underline{\rho}\underline{u}\, \underline{v}(s(x)-b(x_{*}))+\dfrac{b'(x_{*})H(x_{*})}{\sqrt{1+b'(x_{*})^{2}}}\ge0,\quad \forall x\ge x_*,
\end{equation}
and the equality holds only at $x=x_*.$ The equality holds for all $x\ge x_*$ in \eqref{430add} if $\underline{v}=0$ and $b'(x_*)=0$.

From \eqref{4.18}, \eqref{426} and \eqref{427}, \eqref{430add}, we have shown
\begin{equation}\label{429}
s'(x)\ge0\quad\qquad \forall x\ge x_*.
\end{equation}
Next, to prove
\begin{equation}\label{430}
s'(x)\le\frac{\underline{v}}{\underline{u}},\quad\qquad \forall x\ge x_*,
\end{equation}
by \eqref{4.18}, we shall verify
$$s(x)-b(x_*)+\frac{H(x_*)}{\sqrt{1+b'(x_*)}}(1-b'(x_*)\frac{\underline{u}}{\underline{v}})\ge0.$$
It holds at $x=x_*$ and for $x>x_*$ it follows from \eqref{429}. Thus the entropy condition is satisfied by the free layer.

Finally we need to check that $d(x)$ defined in \eqref{423} is positive on $[x_*, \infty)$. Note that $d(x_*)=b(x_*)>0$, and $d'(x)=\underline{\rho}\underline{v}+(1-\underline{\rho}\underline{u})s'(x),$ which is nonnegative, since if $\underline{\rho}\underline{u}\le1$, we use \eqref{429}, and otherwise, by \eqref{430}, $d'(x)\ge\underline{\rho}\underline{v}+(1-\underline{\rho}\underline{u})\frac{\underline{v}}{\underline{u}}=\frac{\underline{v}}{\underline{u}}\ge b'(x_*)\ge0.$ We therefore have a well-defined measure solution of Problem 3 given by \eqref{4.16}-\eqref{4.17}\eqref{4.19}, which also satisfies the entropy condition.
\end{proof}

 This finishes proof of 1) in Theorem \ref{thm1.3}.

\subsection{Construction of measure solutions containing vacuum}\label{sec42}

In the previous subsection we consider the case that the jet impinges on the free layer. It may happen that there is vacuum between the free layer and the pressureless jet. In this section we consider this case. It is known that the only classical discontinuity connecting vacuum and pressureless jet is a contact discontinuity \cite{Cheng2011Delta}.

Suppose the free layer is
$$S_{u}\triangleq\{(x,y)\in \Bbb R^{2}:~x\geq x_{*},\ y=h(x)\},$$
and the contact discontinuity is
$$S_{d}\triangleq\{(x,y)\in \Bbb R^{2}:~x\geq x_{*},\ y=c(x)\},$$
with $y=h(x), \ y=c(x)$  two functions to be determined.
The domain we consider is then
$$\widetilde{\Omega}\triangleq\Omega_{1}\cup\Omega_{2}\cup\Omega_{3},$$
where
\begin{eqnarray*}
&&\Omega_{1}\triangleq\{(x,y)\in \Bbb R^{2}:~0<x\leq x_{*},\ y>b(x)  \ \text{and} \ \mathnormal{x>x_{*},\ y>h(x)}\},\\
&&\Omega_{2}\triangleq\{(x,y)\in \Bbb R^{2}:~x>x_{*},\ y<c(x)\},\\
&&\Omega_{3}\triangleq\{(x,y)\in  \Bbb R^{2}:~x>x_{*},\ c(x)<y<h(x)\}
\end{eqnarray*}
represent the region above $S_{u}$, below $S_{d}$, and between $S_{u}$ and $S_{d}$, respectively.
Above $S_{u}$, there is uniform limiting hypersonic flow given by \eqref{1.6};
The uniform jet below  $S_{d}$ is given by \eqref{1.10}, with pressure $\underline{p}=0$.
The region $\Omega_{3}$ lies between $S_{u}$ and $S_{d}$ is vacuum.

In the following we construct a measure solution to Problem 3 with the above structure.
Denote $\widetilde{W}=W_{f}\cup S_{u}$. Let
\begin{eqnarray}
&&\begin{aligned}
&m^{0}=\mathsf{I}_{\Omega_{1}}\mathcal{L}^{2}+
\underline{\rho}\underline{u}\mathsf{I}_{\Omega_{2}}\mathcal{L}^{2}+w_{m}^{0}(x)\delta_{\widetilde{W}}
=\mathsf{I}_{\Omega_{1}}\mathcal{L}^{2}+\underline{\rho}\underline{u}\mathsf{I}_{\Omega_{2}}\mathcal{L}^{2}+w_{m}^{0}(x)\delta_{W_{f}}
+\widetilde{w_{m}^{0}}(x)\delta_{S_{u}},
\\  &n^{0}=\underline{\rho}\underline{v}\mathsf{I}_{\Omega_{2}}\mathcal{L}^{2}+w_{n}^{0}(x)\delta_{\widetilde{W}}
=\underline{\rho}\underline{v}\mathsf{I}_{\Omega_{2}}\mathcal{L}^{2}+w_{n}^{0}(x)\delta_{W_{f}}+\widetilde{w_{n}^{0}}(x)\delta_{S_{u}},
\end{aligned}\quad\label{4.25}\\
&&\begin{aligned}\label{4.28}
&m^{1}=\mathsf{I}_{\Omega_{1}}\mathcal{L}^{2}+\underline{\rho}\underline{u}^{2}\mathsf{I}_{\Omega_{2}}\mathcal{L}^{2}
+w_{m}^{1}(x)\delta_{\widetilde{W}}=\mathsf{I}_{\Omega_{1}}\mathcal{L}^{2}+\underline{\rho}\underline{u}^{2}\mathsf{I}_{\Omega_{2}}
\mathcal{L}^{2}+w_{m}^{1}(x)\delta_{W_{f}}+\widetilde{w_{m}^{1}}(x)\delta_{S_{u}}, \\
&n^{1}=\underline{\rho}\underline{u} \ \underline{v}\mathsf{I}_{\Omega_{2}}\mathcal{L}^{2}+w_{n}^{1}(x)\delta_{\widetilde{W}}
=\underline{\rho}\underline{u} \ \underline{v}\mathsf{I}_{\Omega_{2}}\mathcal{L}^{2}+w_{n}^{1}(x)\delta_{W_{f}}
+\widetilde{w_{n}^{1}}(x)\delta_{S_{u}},  \ \ \  \wp=0,
\end{aligned}\quad\\
&&\begin{aligned}\label{4.31}
&m^{2}=\underline{\rho}\underline{u}\, \underline{v}\mathsf{I}_{\Omega_{2}}\mathcal{L}^{2}+w_{m}^{2}(x)\delta_{\widetilde{W}}
=\underline{\rho}\underline{u}\, \underline{v}\mathsf{I}_{\Omega_{2}}\mathcal{L}^{2}+w_{m}^{2}(x)\delta_{W_{f}}+\widetilde{w_{m}^{2}}(x)\delta_{S_{u}},
\\ &n^{2}=\underline{\rho}\underline{v}^{2}\mathsf{I}_{\Omega_{2}}\mathcal{L}^{2}+w_{n}^{2}(x)\delta_{\widetilde{W}}
=\underline{\rho}\underline{v}^{2}\mathsf{I}_{\Omega_{2}}\mathcal{L}^{2}+w_{n}^{2}(x)\delta_{W_{f}}+\widetilde{w_{n}^{2}}(x)\delta_{S_{u}},
\end{aligned}\\
&&\begin{aligned}
&m^{3}=E_{0}\mathsf{I}_{\Omega_{1}}\mathcal{L}^{2}+\underline{\rho}\underline{u}\, \underline{E}\mathsf{I}_{\Omega_{2}}\mathcal{L}^{2}+w_{m}^{3}(x)\delta_{W_{f}}
+\widetilde{w_{m}^{3}}(x)\delta_{S_{u}},\\
& n^{3}=\underline{\rho}\underline{v}\ \underline{E}\mathsf{I}_{\Omega_{2}}\mathcal{L}^{2}+w_{n}^{3}(x)\delta_{W_{f}}+\widetilde{w_{n}^{3}}(x)\delta_{S_{u}}.
\end{aligned}
\end{eqnarray}
Here $\widetilde{w_{m}^{i}}(x), \ \widetilde{w_{n}^{i}}(x)\ (i=0, 1, 2, 3)$ are unknown functions.

\medskip

Like before, substituting \eqref{4.25} into \eqref{4.1}, some direct calculation yields
\begin{equation}
\begin{aligned}
&\widetilde{w_{m}^{0}}(x_{*})\sqrt{1+h'(x_{*})^{2}}=w_{m}^{0}(x_{*})\sqrt{1+b'(x_{*})^{2}}=b(x_{*}), \\
&\dfrac{\dd(\widetilde{w_{m}^{0}}(x)\sqrt{1+h'(x)^{2}})}{\dd x}=h'(x), \  \ \
\widetilde{w_{n}^{0}}(x)=h'(x)\widetilde{w_{m}^{0}}(x), \ \ \underline{v}=c'(x)\underline{u}.
\end{aligned}
\end{equation}
So we get
\begin{equation}\label{436}
\begin{aligned}
&\widetilde{w_{m}^{0}}(x)=\dfrac{h(x)}{\sqrt{1+h'(x)^{2}}}, \ \ \  \widetilde{w_{n}^{0}}(x)=\dfrac{h'(x)h(x)}{\sqrt{1+h'(x)^{2}}},  \ \ \
\underline{v}=c'(x)\underline{u},
\end{aligned}
\end{equation}
which imply particularly the slip condition on the free layer and contact discontinuity.
Similarly we also have
\begin{equation}
\widetilde{w_{m}^{3}}(x)=\dfrac{E_{0}h(x)}{\sqrt{1+h'(x)^{2}}},\quad
{\widetilde{w_{n}^{3}}(x)}
=\dfrac{E_{0}h(x)h'(x)}{\sqrt{1+h'(x)^{2}}}.
\end{equation}
By \eqref{4.28} and \eqref{4.2} we find
\begin{equation}\label{4.30}
\begin{aligned}
&\widetilde{w_{m}^{1}}(x_{*})\sqrt{1+h'(x_{*})^{2}}=w_{m}^{1}(x_{*})\sqrt{1+b'(x_{*})^{2}}=\dfrac{H(x_{*})}{\sqrt{1+b'(x_{*})^{2}}},
 \\ &\widetilde{w_{n}^{1}}(x)=h'(x)\widetilde{w_{m}^{1}}(x), \ \  \  \dfrac{\dd(\widetilde{w_{m}^{1}}(x)\sqrt{1+h'(x)^{2}})}{\dd x}=h'(x),\ \ \
\underline{v}=c'(x)\underline{u}.
\end{aligned}
\end{equation}
One then solves that
\begin{equation}
\begin{aligned}
\widetilde{w_{m}^{1}}(x)=\dfrac{h(x)-b(x_{*})+\dfrac{H(x_{*})}{\sqrt{1+b'(x_{*})^{2}}}}{\sqrt{1+h'(x)^{2}}},\ \ \
\underline{v}=c'(x)\underline{u}.
\end{aligned}
\end{equation}
Substituting \eqref{4.31} into \eqref{4.3},  one deduces that
\begin{equation}
\begin{aligned}
&\widetilde{w_{m}^{2}}(x_{*})\sqrt{1+h'(x_*)^{2}}=w_{m}^{2}(x_{*})\sqrt{1+b'(x_*)^{2}}=\dfrac{b'(x_{*})H(x_{*})}{\sqrt{1+b'(x_{*})^{2}}},
 \\  &\widetilde{w_{n}^{2}}(x)=h'(x)\widetilde{w_{m}^{2}}(x),\ \ \   \dfrac{\dd(\widetilde{w_{m}^{2}}(x)\sqrt{1+h'(x)^{2}})}{\dd x}=0, \ \ \
\underline{v}=c'(x)\underline{u}.
\end{aligned}
\end{equation}
Therefore
\begin{equation}
\begin{aligned}
&\widetilde{w_{m}^{2}}(x)=\dfrac{\dfrac{b'(x_{*})H(x_{*})}{\sqrt{1+b'(x_{*})^{2}}}}{\sqrt{1+h'(x)^{2}}},  \ \ \
\underline{v}=c'(x)\underline{u}.
\end{aligned}
\end{equation}
Thanks to \eqref{2.8}, we see
\begin{equation}\label{4.36}
u|_{S_{u}}=\dfrac{h(x)-b(x_{*})+\dfrac{H(x_{*})}{\sqrt{1+b'(x_{*})^{2}}}}{h(x)},
\quad v|_{S_{u}}=\dfrac{\dfrac{b'(x_{*})H(x_{*})}{\sqrt{1+b'(x_{*})^{2}}}}{h(x)}.
\end{equation}
Therefore
\begin{equation}\label{443}
\begin{aligned}
&u=\mathsf{I}_{\Omega_{1}}+\underline{u}\mathsf{I}_{\Omega_{2}}+\dfrac{H(x)}{b(x)\sqrt{1+b'(x)^{2}}}\mathsf{I}_{W_{f}}+\dfrac{h(x)-b(x_{*})
+\dfrac{H(x_{*})}{\sqrt{1+b'(x_{*})^{2}}}}{h(x)}\mathsf{I}_{S_{u}},\\
 &v=\underline{v}\mathsf{I}_{\Omega_{2}}+\dfrac{b'(x)H(x)}{b(x)\sqrt{1+b'(x)^{2}}}\mathsf{I}_{W_{f}}
 +\dfrac{\dfrac{b'(x_{*})H(x_{*})}{\sqrt{1+b'(x_{*})^{2}}}}{h(x)}\mathsf{I}_{S_{u}},\\
 &E=E_0\mathsf{I}_{\Omega_{1}}+\underline{E}\mathsf{I}_{\Omega_{2}}+E_0\mathsf{I}_{W_{f}}+E_0\mathsf{I}_{S_{u}}
\end{aligned}
\end{equation}
and the measure of density is
\begin{equation}\label{444}
\begin{aligned}
&\varrho=\mathsf{I}_{\Omega_{1}}\LL^{2}+\underline{\rho}\mathsf{I}_{\Omega_{2}}\LL^{2}+\dfrac{(b(x))^{2}}{H(x)}
\mathsf{\delta}_{W_{f}}+\dfrac{(h(x))^{2}}{\sqrt{1+h'(x)^{2}}[h(x)-b(x_{*})+\dfrac{H(x_{*})}{\sqrt{1+b'(x_{*})^{2}}}]}\delta_{S_{u}}.\\
\end{aligned}
\end{equation}

\medskip
Next we solve $h(x)$.
Using the slip condition implied by \eqref{436} on $S_{u}$, we have
\begin{equation}\label{445}
\left \{
\begin{aligned}
&\dfrac{b'(x_{*})H(x_{*})}{\sqrt{1+b'(x_{*})^{2}}}=h'(x)[h(x)-b(x_{*})+\dfrac{H(x_{*})}{\sqrt{1+b'(x_{*})^{2}}}],\\
&h(x_{*})=b(x_{*}).
\end{aligned} \right.
\end{equation}
The solution is
\begin{equation}\label{4.39}
\begin{aligned}
&h(x)=\sqrt{\dfrac{2H(x_{*})b'(x_{*})}{\sqrt{1+b'(x_{*})^{2}}}(x-x_{*})+(\dfrac{H(x_{*})}{\sqrt{1+b'(x_{*})^{2}}})^{2}}+b(x_{*})-\dfrac{H(x_{*})}
{\sqrt{1+b'(x_{*})^{2}}}.
\end{aligned}
\end{equation}
It is the same as \eqref{3.31}, while representing the situation that below the free layer is vacuum.

From \eqref{436} we also infer that  $\underline{v}=c'(x)\underline{u}$.
So $S_{d}$, the contact discontinuity, is a straight line
\begin{equation}\label{4.42}
\begin{aligned}
&c(x)=\displaystyle\dfrac{\underline{v}}{\underline{u}}(x-x_{*})+b(x_{*}).
\end{aligned}
\end{equation}

From \eqref{4.39} and \eqref{4.42}, we infer that if the pressureless jet satisfies the requirement that $\underline{v}\leq 0$, the for all $x\geq x_{*}$, $h'(x)\ge0$ and $c'(x)\le0$, hence  $h(x)\geq c(x)$, and the vacuum between the limiting hypersonic flow  and the pressureless jet is unbounded, see Figure \ref{fig7}.

\vspace{-3.5cm}
\begin{figure}[H]
\centering
\includegraphics[width=6in]{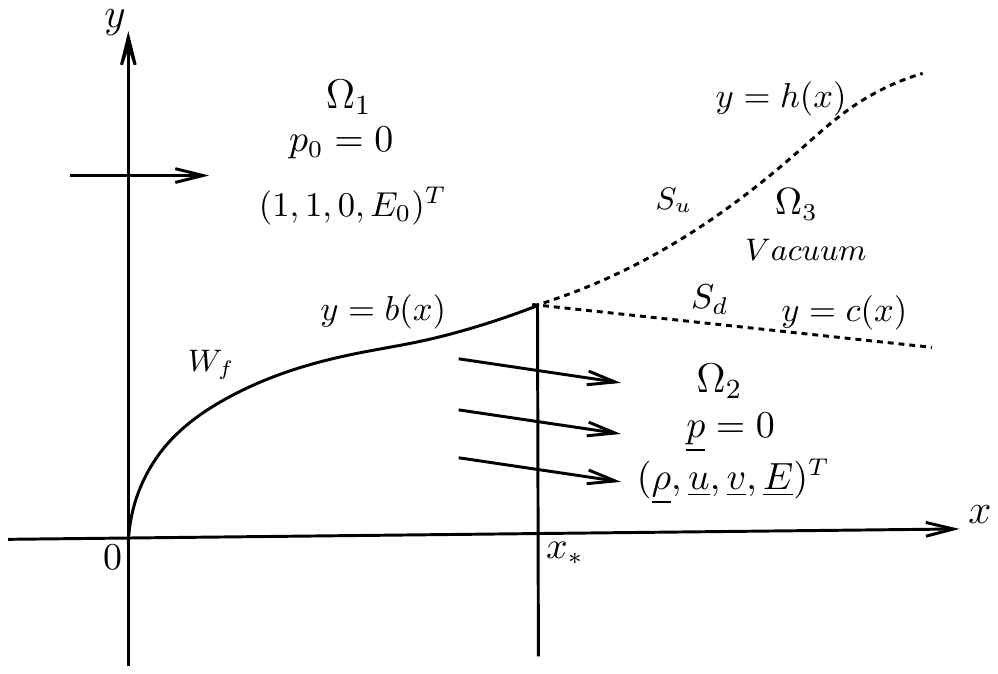}\\
\vspace{-13cm}\caption{Unbounded vacuum between free layer and pressureless jet $(\underline{v}\leq 0)$.}\label{fig7}
\end{figure}

\medskip
What happens if $0< \displaystyle\frac{\underline{v}}{\underline{u}}<b'(x_{*})$? Simple computation shows that on $$x_{*}\leq x\le x^{\vartriangle}\triangleq x_{*}+\displaystyle\frac{2(\underline{u}b'(x_{*})-\underline{v})\underline{u}}{\underline{v}^{2}}\dfrac{H(x_{*})}{\sqrt{1+b'(x_{*})^{2}}},$$ it holds that $h(x)\geq c(x)$. The free layer meets the contact discontinuity at the point $(x^{\vartriangle}, h(x^{\vartriangle}))$, with $h(x^{\vartriangle})=\displaystyle\frac{2(\underline{u}b'(x_{*})-\underline{v})}{\underline{v}}\dfrac{H(x_{*})}{\sqrt{1+b'(x_{*})^{2}}}+b(x_{*}).$ Then from \eqref{445} we have $h'(x^{\vartriangle})=\displaystyle\frac{b'(x_{*})\underline{v}}{2\underline{u}b'(x_{*})-\underline{v}}$.

It is easy to show that for $0< \displaystyle\frac{\underline{v}}{\underline{u}}<b'(x_{*})$, it holds that  $h'(x^{\vartriangle})\leq \displaystyle\frac{\underline{v}}{\underline{u}}.$ Therefore considering a problem like Problem 3 with $x=x_*$ replaced by  $x=x^{\vartriangle}$, we have a case studied in Section \ref{sec41}: {\it the free layer absorbs the contact discontinuity}.  See Figure \ref{fig9}.

Using the method in Section \ref{sec41}, we could determine the free layer starting from the colliding point as follows.

\medskip

(1) If $\underline{\rho}\underline{u}^{2}\ne1$, then for $x\ge x^{\vartriangle}$,
\begin{equation}\label{451}
\begin{aligned}
s(x)=&\dfrac{(1-\underline{\rho}\underline{u}^{2})h(x^{\vartriangle})-\underline{\rho}\underline{u}\, \underline{v}(x-x^{\vartriangle})-\displaystyle\frac{2\underline{u}b'(x_{*})
-\underline{v}}{\underline{v}}\dfrac{H(x_{*})}{\sqrt{1+b'(x_{*})^{2}}}}{1-\underline{\rho}\underline{u}^{2}}{+\dfrac{\sqrt{(*)}}{1 -\underline{\rho}\underline{u}^{2}}},
\end{aligned}
\end{equation}
where
\begin{equation*}
\begin{aligned}
   (*)=&{\underline{\rho}\underline{v}^{2}(x-x^{\vartriangle})^{2}+}2[\underline{\rho}\underline{u}\, \underline{v}\cdot\displaystyle\frac{2\underline{u}b'(x_{*})-\underline{v}}{\underline{v}}
+(1-\underline{\rho}\underline{u}^{2})b'(x_{*})]\dfrac{H(x_{*})}{\sqrt{1 +b'(x_{*})^{2}}}(x-x^{\vartriangle})\\
&+(\displaystyle\frac{2\underline{u}b'(x_{*})-
\underline{v}}{\underline{v}}\cdot\dfrac{H(x_{*})}{\sqrt{1+b'(x_{*})^{2}}})^{2}.
\end{aligned}
  \end{equation*}

(2) If $\underline{\rho}\underline{u}^{2}=1$, then for $x\ge x^{\vartriangle}$,
\begin{equation}\label{452}
\begin{aligned}
s(x)=&\dfrac{\underline{\rho}\underline{v}^{2}(x-x^{\vartriangle})^{2}+2[\dfrac{b'(x_{*})H(x_{*})}{\sqrt{1 +b'(x_{*})^{2}}}+ \underline{\rho}\underline{u}\, \underline{v}h(x^{\vartriangle})](x-x^{\vartriangle})+C}{2\underline{\rho}\underline{u}\, \underline{v}(x-x^{\vartriangle})+2\cdot\displaystyle\frac{2\underline{u}b'(x_{*})
-\underline{v}}{\underline{v}}\dfrac{H(x_{*})}{\sqrt{1+b'(x_{*})^{2}}}},
\end{aligned}
\end{equation}
where
$$C=2h(x^{\vartriangle})\cdot\displaystyle\frac{2\underline{u}b'(x_{*})
-\underline{v}}{\underline{v}}\dfrac{H(x_{*})}{\sqrt{1+b'(x_{*})^{2}}}.$$

In summary, we proved the following lemma.

\begin{lemma}\label{lem4.2}
\noindent 1) If $\underline{v}\leq 0$, then Problem 3 has a measure solution that consists of three piecewise constant states, with the one lying middle is an unbounded vacuum.

\noindent 2) If $0< \displaystyle\frac{\underline{v}}{\underline{u}}<b'(x_{*})$, then Problem 3 has a measure solution consists of three piecewise constant states, with the one in the middle a finite vacuum bounded by the curves given by \eqref{4.39}\eqref{4.42}, lying in $\{x_{*}\leq x\leq x^{\vartriangle}\}$.
In particular, if $x\geq x^{\vartriangle}$, there is no vacuum and the solution is two piecewise constant.
\end{lemma}

This completes the proof of Theorem \ref{thm1.3}.

\vspace{-3.5cm}
\begin{figure}[H]
\centering
\includegraphics[width=6in]{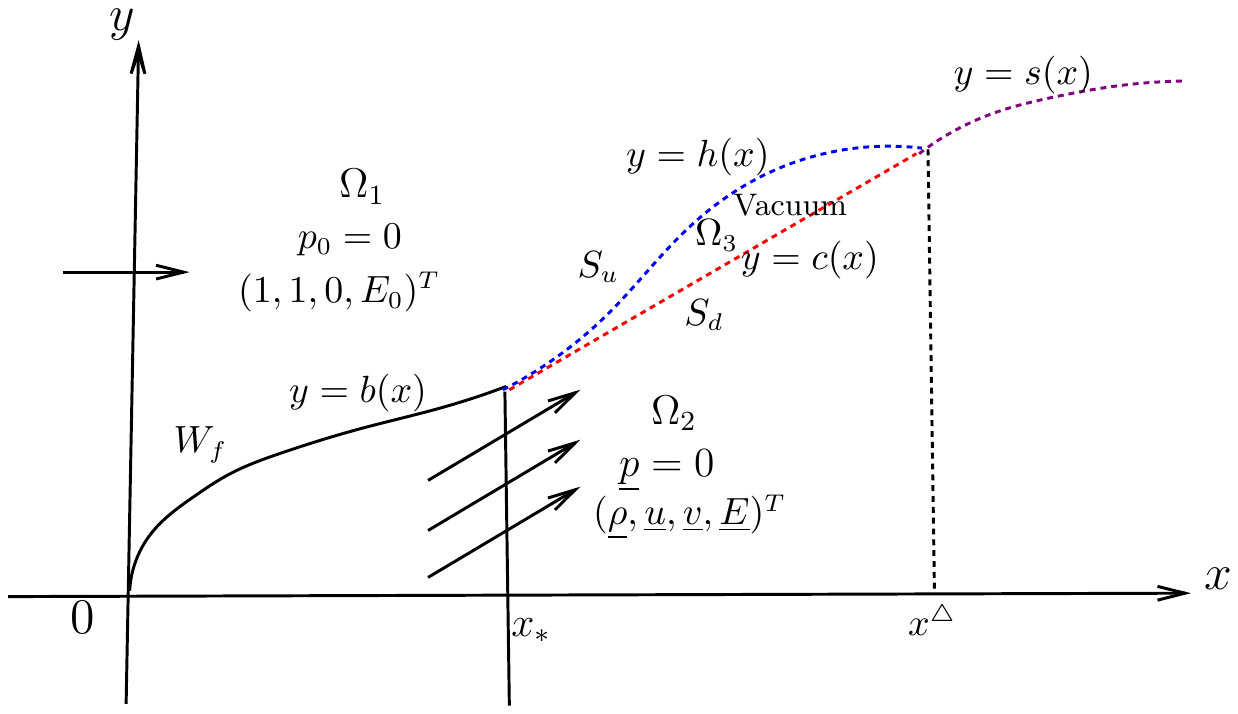}\\
\vspace{-12cm}\caption{Bounded vacuum between the limiting hypersonic flow and pressureless jet when $0<\displaystyle\frac{\underline{v}}{\underline{u}}<b'(x_{*})$.}\label{fig9}
\end{figure}

\section{Discussions}\label{sec5}
Considering hypersonic flows, Louie and Ockendon wrote in \cite[p.121]{LO1991}: ``Although inviscid models have limited practical value, it is important to understand them  as well as possible if theoretical progress is to be made with more complicated models for real gases."  This paper maybe considered a part of progress to understand inviscid hypersonic flows. We studied three typical aerodynamical problems on limiting hypersonic flows by using a concept of Radon measure solutions of compressible Euler equations we proposed. We proved succinctly the well-known Newton-Busemann pressure law that was firstly given by Busemann in 1930s, and constructed various flow fields with free layers, and discovered blow up of some measure solutions. These results demonstrate the power of our concept of measure solutions.
However, the uniqueness and stability of the measure solutions we constructed remain interesting open problems. To solve them some meticulous refinement of our definition
might be necessary.

\medskip

To solve Problem 3, noticing that $\widetilde{\Omega}_{1}\cup \widetilde{\Omega}_{2}\triangleq\{(x,y)\in \Bbb R^{2}:0< x\leq x_{*}, \ y>b(x)\}\cup\{(x,y)\in \Bbb R^{2}:x> x_{*},\ y\in\Bbb R\},$
using solutions of Problem 1 in $\widetilde{\Omega}_{1}$, we only need to solve a special boundary value  problem of
\eqref{1.4} in $\widetilde{\Omega}_{2}$, subjected to boundary conditions given on the line $\{x=x_*\}$:
\begin{equation}\label{3.34}
\begin{aligned}
&\varrho(x_{*},y)=\mathsf{I}_{\{y>b(x_{*})\}}\LL^1+\dfrac{(b(x_{*}))^{2}}{H(x_{*})}\delta_{\{y=b(x_{*})\}} +\underline{\rho}\mathsf{I}_{\{y<b(x_{*})\}}\LL^1,\\
&u(x_{*},y)=\mathsf{I}_{\{y>b(x_{*})\}}+\dfrac{H(x_{*})}{b(x_{*})\sqrt{1+b'(x_{*})^{2}}} \mathsf{I}_{\{y=b(x_{*})\}}+\underline{u}\mathsf{I}_{\{y<b(x_{*})\}},\\
&  v(x_{*},y)=\dfrac{b'(x_{*})H(x_{*})}{b(x_{*})\sqrt{1+b'(x_{*})^{2}}}\mathsf{I}_{\{y=b(x_{*})\}} +\underline{v}\mathsf{I}_{\{y<b(x_{*})\}},\\
&E(x_{*},y)=E_{0}\mathsf{I}_{\{y>b(x_{*})\}}+E_{0}\mathsf{I}_{\{y=b(x_{*})\}} +\underline{E}\mathsf{I}_{\{y<b(x_{*})\}},
\end{aligned}
\end{equation}
where $\LL^1$ is the Lebesgue measure on real line $\mathbb{R}$.
Unlike the classical Riemann problems, there is a Dirac measure supported on the discontinuity point $(x_*,b(x_*))$ for the density. We call such problems as {\it Singular Riemann Problems}.

We know that there have been many studies on singular Riemann problems from the mathematical point of view (see, for example, \cite{Guo2018The, Wang2012THE, Wang2016The,Yang2007The} and references therein). Our studies of limiting hypersonic flows show that such problems are also models of significant physical problems, hence they required a systematic investigation.

Comparing to classical Riemann problems, solutions to singular Riemann problems may no longer be self-similar and it may contain rather complex wave structures (such as bounded vacuum in Problem 3). If we consider the interactions of limiting hypersonic flows with supersonic polytropic gas jets, the situation is more complicated. We will report the results in another paper.

\section*{Acknowledgements}
The research of Aifang Qu is supported by National Natural
Science Foundation of China (NNSFC) under Grants No.11571357 and
No.11871218;
Hairong Yuan is supported by NNSFC under Grant No.11871218, and by
Science and Technology Commission of Shanghai Municipality (STCSM)
under Grant No.18dz2271000.

\end{document}